%% file: gw_complex_v16.tex
\documentclass[a4paper]{amsart}

\input{amsart_header_18_05}

\usepackage{enumitem} 
\usepackage[foot]{amsaddr}  

\numberwithin{equation}{section} 

\newcommand{\bDer}[2][]{\calD^{\mathrm{b}}_{#1}(#2)} 
\newcommand{\bHot}[2][]{\calK^{\mathrm{b}}_{#1}(#2)} 

\newcommand{\rMod}[1]{\calM({#1})}  
\newcommand{\rmod}[1]{\calM_{\mathrm{f}}(#1)}  
\newcommand{\fl}[1]{\calM_{\mathrm{fl}}(#1)} 
\newcommand{\rproj}[1]{{\calP({#1})}}              

\newcommand{\Herm}[2][]{{\mathcal{H}}^{#1}(#2)} 
\newcommand{\tHerm}[2][]{{\tilde{\mathcal{H}}}^{#1}(#2)} 

\DeclareMathOperator{\loc}{loc} 
\DeclareMathOperator{\res}{res} 

\newcommand{\dd}{{\mathrm{d}}} 

\newcommand{\aGW}[1]{{\mathcal{GW}}^{#1}_+} 
\newcommand{\GW}[1]{{\mathcal{GW}}^{#1}} 
\newcommand{\iaGW}[1]{{   \mathcal{B}}^{#1}_+} 
\newcommand{\iGW}[1]{{  \mathcal{B}}^{#1}} 

\DeclareMathOperator{\Tot}{Tot} 

\newcommand{\RD}{\Delta} 

\title[On The Gersten--Witt Complex]{On The Gersten--Witt Complex of an Azumaya Algebra with Involution}

\author{Uriya A.\ First$^*$}
\address{$^*$Department of Mathematics, University of Haifa}
\email{uriya.first@gmail.com}

\begin{document}

\maketitle

\begin{abstract}
Let $(A,\sigma)$ be an Azumaya algebra with involution over a regular ring $R$.
We prove that  the Gersten--Witt complex
of   $(A,\sigma)$ defined by Gille is isomorphic to the Gersten--Witt complex
of $(A,\sigma)$ defined by  Bayer-Fluckiger, Parimala and the author.
Advantages of both constructions are used to show that the Gersten--Witt complex
is exact when $\dim R\leq 3$, $\ind A\leq 2$ and $\sigma$ is orthogonal or symplectic.
This means that the Grothendieck--Serre conjecture holds 
for the   group $R$-scheme
of $\sigma$-unitary elements in $A$ under the same hypotheses;  
$R$ is not required to contain a field.
\end{abstract}

\section{Introduction}

Let $R$ be a regular   domain with fraction field $K$
such that $2\in\units{R}$.  
The Gersten conjecture for Witt groups, suggested by
Pardon  \cite{Pardon_1982_Gersten_conjecture}, predicts the existence  of   
a cochain complex of Witt groups
of the form
\[
0\to W(R)\xrightarrow{\dd_{-1}} W(K)
\xrightarrow{\dd_0} \bigoplus_{\frakp\in R^{(1)}} W(k(\frakp))
\xrightarrow{\dd_1} \bigoplus_{\frakp\in R^{(2)}} W(k(\frakp))
\xrightarrow{\dd_2} \cdots,
\]
called a \emph{Gersten--Witt  complex} for $R$, which is exact when $R$ is local. 
Here, $R^{(e)}$ is
the set of height-$e$ primes in $R$ and
$k(\frakp)$ is the fraction field of $R/\frakp$.
The map $\dd_{-1}$ is   base-change from $R$ to $K$,
and the $\frakp$-component of $\dd_1$ is   a \emph{second residue map}
$W(K)\to W(k(\frakp))$ in the sense of \cite[p.~209]{Scharlau_1985_quadratic_and_hermitian_forms}.

Pardon \cite{Pardon_1982_relation_between_Witt_groups_zero_cycles} gave a construction of a Gersten--Witt complex for any regular $R$, and proved
its exactess if $R$ is local of dimension $4$ or less.
Balmer and Walter \cite{Balmer_2002_Gersten_Witt_complex} later gave another construction
using \emph{triangulated hermitian categories}, and also proved its exactness
if $R$ is local of dimension $\leq 4$. 
Their approach revealed a spectral sequence underlying the Gersten--Witt complex, which
led to further positive results about the exactness:
It was established when   $R$ is semilocal and contains a field by 
Balmer, Gille, Panin and Walter \cite{Balmer_2002_Gersten_conj_equicharacteristic}
(see also \cite{Balmer_2001_Witt_cohomology}),  
when $R$ is semilocal of dimension $\leq 4$ by Balmer and Preeti \cite{Balmer_2005_shifted_Witt_groups_semilocal},
and when $R$ is local and unramified of mixed characteristic by 
Jacobson \cite{Jacobson_2018_cohomological_invariants}.
Another construction, applying
if $R$ is essentially of finite type over a field,
appears in  Schmid \cite{Schmid_1997_PhD}.

Let $(A,\sigma)$ be an Azumaya algebra with involution
over $R$ 
(see~\ref{subsec:Azumaya}) and let $\veps\in\{\pm 1\}$.
By introducing a modification to the Balmer--Walter construction,
Gille \cite{Gille_2007_hermitian_GW_complex_I}, \cite{Gille_2009_hermitian_GW_complex_II}
defined a Gersten--Witt complex    of  Witt groups of $\veps$-hemritian forms over $(A,\sigma)$,
and proved its exactness when $R$ is  local and contains a field \cite[Theorem~7.7]{Gille_2013_coherent_herm_Witt_grps}. Broadly speaking, the modification
to \cite{Balmer_2002_Gersten_Witt_complex} 
consisted of replacing the bounded derived category of finite projective $R$-modules and the duality
$\Hom(-,R)$
with
a suitable
subcategory of the bounded derived category of quasi-coherent sheaves  over $\Spec R$ and the duality induced by a dualizing complex.
When $(A,\sigma,\veps)=(R,\id_R,1)$, Gille's Gersten--Witt complex
is isomorphic to the one defined by Balmer and Walter.

Recently, Bayer-Fluckiger, Parimala and the author \cite{Bayer_2019_Gersten_Witt_complex_prerprint} introduced another construction
of a Gersten--Witt complex of $\veps$-hemritian forms over $(A,\sigma)$,
and proved it is exact when $R$ is semilocal and  $\dim R\leq 2$ ($R$ is not required to contain  a field).
In contrast with \cite{Balmer_2002_Gersten_Witt_complex}, \cite{Gille_2007_hermitian_GW_complex_I}, \cite{Gille_2009_hermitian_GW_complex_II}, this 
construction 
does not require the use of triangulated hermitian categories.

The main purpose of this work is to prove that the Gersten--Witt complexes
defined by Gille and by Bayer-Fluckiger, Parimala and the author are in fact isomorphic.
In the process, we  streamline  
the Balmer--Walter construction of the Gersten--Witt complex, particularly the \emph{d\'evissage}, 
and explain how it can be applied to Azumaya algebras with involution
without using a dualizing complex as in \cite{Gille_2007_hermitian_GW_complex_I}, \cite{Gille_2009_hermitian_GW_complex_II}.
This is made possible thanks to a new method of transfer in hermitian categories, which is of independent interest;
see Section~\ref{sec:transfer}.

The isomorphism between the constructions means that   tools
developed for individual constructions can be applied together.
Specifically, it allows us to use
the spectral sequence underlying Gille's Gersten--Witt complex
\cite{Gille_2007_hermitian_GW_complex_I}, \cite{Gille_2009_hermitian_GW_complex_II}
together with the $8$-periodic exact sequence of \cite{First_2022_octagon},
which is compatible with the Gersten--Witt complex of   \cite{Bayer_2019_Gersten_Witt_complex_prerprint}.
By putting both of these together, we prove that the Gersten--Witt complex
of $(A,\sigma)$ is exact when $R$ is semilocal of dimension $\leq 3$,
the index of $A$ is at most $2$
and $\sigma$ is orthogonal or symplectic; see Theorem~\ref{TH:exactness-dim-three-ind-two}.

This in turn establishes some open  cases of the famous \emph{Grothendieck--Serre conjecture},
which predicts that the map $\HH^1_{\et}(R,\bfG)\to \HH^1_{\et}(K,\bfG)$
has trivial kernel for every regular local ring $R$ and every reductive (connected) group scheme
$\bfG\to \Spec R$. See \cite[\S5]{Panin_2018_Grothendieck_Serre_conj_survey} 
and \cite{Cesnavicius_2020_Grothendieck_Serre_split_unramified_preprint} for surveys
of the many works discussing this conjecture. Our Theorem~\ref{TH:exactness-dim-three-ind-two}
and \cite[Proposition~8.7]{First_2022_octagon} imply:

\begin{thm}
	Let $R$ be a regular semilocal domain such that $\dim R\leq 3$
	and $2\in\units{R}$,
	and let $K$ be the fraction field of $R$.
	Let $(A,\sigma)$ be an Azumaya $R$-algebra with involution
	such that $\ind A\leq 2$ and $\sigma$ is orthogonal or symplectic.
	Let $\bfG$ denote the neutral component of the group $R$-scheme of elements
	$a\in A$ satisfying $a^\sigma a=1$. 
	Then $\HH^1_{\et}(R,\bfG)\to \HH^1_{\et}(K,\bfG)$
	is injective.
\end{thm}

The paper is organized as follows:
Section~\ref{sec:Prelim} recalls   hermitian categories,
Azumaya algebras with involution
and some facts about $\Ext$-groups.
In Section~\ref{sec:GW-second-res} we recall the Gersten--Witt complex 
defined by Bayer-Fluckiger, Parimala and the author,
and in Section~\ref{sec:GW-Balmer} we recall the constructions
of Balmer--Walter and Gille. The proof that the constructions 
yield  isomorphic cochain complexes is given in Sections~\ref{sec:transfer}--\ref{sec:proof-of-dd-desc}:
In Section~\ref{sec:transfer}, we introduce a variant of transfer in hermitian categories,
which is applied in Section~\ref{sec:GW-propositions}
to define an isomorphism (called \emph{d\'evissage}) between the respective 
terms of both Gersten--Witt complexes. This is shown to be an isomorphism of cohain complexes
in Section~\ref{sec:proof-of-dd-desc}.
Finally, in Section~\ref{sec:exactness-more}, we apply
the isomorphism to establish the exactness of the Gersten--Witt complex
when $\ind A$ and $\dim R$ are small.

\medskip

We thank Stefan Gille for his comments on an earlier version  of this paper.
We also thank the anonymous referee for their comments.

\subsection*{Notation}

Throughout this paper, a ring means a commutative (unital) ring. 
Algebras are unital, but not necessarily commutative. 
{\it We assume
that $2$ is invertible in all rings an algebras.}

Unless otherwise indicated, $R$ denotes a ring. 
Unadorned $\Hom$-groups, $\Ext$-groups and tensors are taken over $R$.
An $R$-ring means a commutative
$R$-algebra. An $R$-algebra with involution   is a pair $(A,\sigma)$
such that $A$ is an $R$-algebra and $\sigma:A\to A$ is an \emph{$R$-linear}
involution. An invertible $R$-module
is a rank-$1$ projective $R$-module.
Given $\frakp\in \Spec R$, we write $k(\frakp)$ for
the fraction ring of $R/\frakp$, and set
$M(\frakp)= M\otimes k(\frakp)$ for any $R$-module $M$.
If $f:M\to N$ is an $R$-module homomorphism, let $f(\frakp)$ denote
$f\otimes\id_{k(\frakp)}:M(\frakp)\to N(\frakp)$.
We also let $\mu_2(R)=\{r\in R\suchthat r^2=1\}$.

Let $A$ be an $R$-algebra.
The center and the group of units of $A$ are denoted $\Cent(A)$ 
and $\units{A}$, respectively.
The category of all (resp.\ finite, finite projective, finite length) right $A$-modules
is denoted $\rMod{A}$ (resp.\ $\rmod{A}$, $\rproj{A}$, $\fl{A}$). Here, 
an $A$-module is said to be finite if it is finitely generated.
The bounded derived category of $\rproj{A}$ is denoted $\bDer{A}$.
Throughout, $\Hom$-groups and $\Ext$-groups taken over $A$
are  taken in the category of \emph{right} $A$-modules.
We write $A_A$ to denote $A$ when regarded as right module over itself.

\section{Preliminaries}
\label{sec:Prelim}

We recall some facts about hermitian categories,
Azumaya algebras with involution and $\Ext$-groups,
setting  notation along the way.
For an extensive discussion,
see \cite[Chapter~II]{Knus_1991_quadratic_hermitian_forms}
and \cite[\S1]{First_2022_octagon}, for instance.

\subsection{Hermitian categories}
\label{subsec:herm-cat}

As usual, a hermitian $R$-category is a triple $(\catC,*,\omega)$
such that $\catC$ is an additive $R$-category, $*:\catC\to \catC$
is a contravariant $R$-linear functor,
and $\omega:\id \to **$ is a natural isomorphism
satisfying $\omega_P^*\circ \omega_{P^*}=\id_{P^*}$ for all $P\in\catC$.
We also say that $(*,\omega)$ is a hermitian structure on $\catC$.

Given $\veps\in \mu_2(R)$,
an $\veps$-hermitian space over $(\catC,*,\omega)$
is a pair $(P,f)$, where $P\in\catC$
and $f:P\to P^*$
is a morphism satisfying $f=\veps f^*\circ \omega_P $.
We say that $(P,f)$ is unimodular if $f$ is an isomorphism.
The category of unimodular $\veps$-hermitian spaces over
$(\catC,*,\omega)$ with isometries as morphisms is denoted
$\Herm[\veps]{\catC,*,\omega}$.

Suppose in addition that $\catC$ is an exact category.
We say that $(\catC,*,\omega)$ is
an \emph{exact hermitian category},
or that $(*,\omega)$ is an \emph{exact hermitian structure} on $\catC$,  if 
$*:\catC\to \catC$ is exact.
In this case, we   define the Witt group of $\veps$-hermitian forms over $(\catC,*,\omega)$, 
denoted $W_\veps(\catC,*,\omega)$ or $W_\veps(\catC)$, as in
\cite[\S1.1]{Balmer_2005_Witt_groups}. That is,
$W_\veps(\catC)$ is the quotient of
the Grothendieck group of $\Herm[\veps]{\catC,*,\omega}$ (relative to orthogonal sum)
by the subgroup generated by \emph{metabolic} hermitian spaces. Here,
a unimodular $\veps$-hermitian space $(P,f)$ is called metabolic if there exists
a short exact sequence $L\embeds P\onto M$ in $\catC$
such that   $L\embeds P$ is the kernel of the composition $P\xrightarrow{f} P^*\onto L^*$. 
The   element of $W_\veps(\catC)$ represented by $(P,f)\in\Herm[\veps]{\catC}$
is  the \emph{Witt class}  of $(P,f)$ and 
denoted $[P,f]$ or $[f]$.

\begin{example}\label{EX:line-bundle}
	Let $(A,\sigma)$ be an $R$-algebra with involution and let $M$ be an invertible
	$R$-module. 
	Then $\sigma$ and $M$  induce  an exact
	hermitian structure $(*,\omega)=(*_{\sigma,M},\omega_{\sigma,M})$
	on $\rproj{A}$  as follows:
	The functor $*$ sends an object $P\in\rproj{A}$ to   $\Hom_A(P, A \otimes M)$ endowed with the \emph{right} 
	$A$-module structure given by $(\phi a)x=a^\sigma(\phi x)$ ($\phi\in P^*$, $a\in A$, $x\in P$),
	and   a morphism $\vphi$ to $\Hom_A(\vphi,A \otimes M)$.
	The isomorphism $\omega_P:P\to P^{**}$
	is given by   $(\omega_Px)\phi=(\phi x)^{\sigma\otimes \id_M}$ 
	($x\in P$, $\phi\in P^*$). 
	The pair $(*,\omega)$ is an exact hermitian structure
	by \cite[Example~1.2]{Bayer_2019_Gersten_Witt_complex_prerprint}.

	Hermitian spaces over $(\rproj{A},*,\omega)$
	and $(A\otimes M,\sigma\otimes\id_M)$-valued hermitian spaces over $(A,\sigma)$ 
	in the sense of \cite[\S1B]{Bayer_2019_Gersten_Witt_complex_prerprint}
	are essentially the same thing. 
	Recall that the latter are pairs 
	$(P,\hat{f})$ consisting of $P\in \rproj{A}$ and a biadditive
	map $\hat{f}:P\times P\to A\otimes M$
	satisfying $\hat{f}(xa,x'a')=a^\sigma \hat{f}(x,x')a'$ and $\hat{f}(x,x')=\veps \hat{f}(x',x)^{\sigma\otimes\id_M}$
	for all $x,x'\in P$, $a,a'\in A$.
	The map corresponding to $f:P\to P^*$ is
	$\hat{f}:P\times P\to A\otimes M$ given by
	$
	\hat{f}(x,y)=(fx)y$. 

	In the sequel, the category $\Herm[\veps]{\catC,*,\omega}$ and the Witt group
	$W_\veps(\catC,*,\omega)$ will be denoted as
	$\Herm[\veps]{A,\sigma;M}$ and $W_\veps(A,\sigma;M)$, respectively.
	We abbreviate
	this to $\Herm[\veps]{A,\sigma}$ and $W_{\veps}(A,\sigma)$ when $M=R$.
\end{example}

Let   $(\catD,\#,\eta)$ be another    hermitian $R$-category  and let $\gamma \in \mu_2(R)$.
Recall that a $\gamma$-hermitian functor from $(\catC,*,\omega)$ to $(\catD,\#,\eta)$
consists of a pair $(F,i)$, where   $F:\catC\to \catD$
is an   additive $R$-functor
and $i: F*\to \# F$ is a natural isomorphism
satisfying
$i_{P^*}\circ F\omega_P = \gamma\cdot i_P^\# \circ \eta_{FP}$
for all $P\in\catC$.
This induces a functor
$
F=\Herm{F,i}:\Herm[\veps]{\catC}\to\Herm[\gamma\veps]{\catD}
$
given by $F(P,f)=(FP,i_P\circ Ff)$ on objects and acting as $F$ on morphisms.
If $F:\catC\to \catD$ is an equivalence, then so is $\Herm{F,i}$,
and $(F,i)$ is called a \emph{$\gamma$-hermitian equivalence}.
If $(\catC,*,\omega)$ and $(\catD,\#,\eta)$   are  exact   hermitian $R$-categories
and $F$ is exact, then we also have an induced group homomorphism
$
F=W(F,i):W_\veps(\catC)\to W_{\gamma\veps}(\catD) 
$,
which is an isomorphism when $F$ is an equivalence of exact categories.

\subsection{Azumaya algebras with involution}
\label{subsec:Azumaya}

Following \cite[\S1.1]{First_2022_octagon},
an $R$-algebra $A$ is called \emph{separable projective} if it is separable
as an $R$-algebra and projective as $R$-module, or equivalently,
if $A$ is Azumaya over its center $\Cent(A)$  and $\Cent(A)$
is finite \'etale over $R$. An $R$-algebra with involution $(A,\sigma)$
is \emph{Azumaya}  if $A$ is separable projective over $R$,
and the structure map $R\to \Cent(A)^{\{\sigma\}}:=\{a\in \Cent(A)\suchthat {a^\sigma=a}\}$
is an isomorphism. 
See \cite[\S1]{First_2022_octagon} for further details.

We will work with separable projective $R$-algebras with involution
throughout, rather than the more restricted class of Azumaya $R$-algebras with involution.
This makes little difference in practice, because if $(A,\sigma)$ is a separable
projective $R$-algebra with involution, then $(A,\sigma)$ is Azumaya over
$\Cent(A)^{\{\sigma\}}$ and $\Cent(A)^{\{\sigma\}}$ is finite \'etale over $R$
\cite[Example~1.20]{First_2022_octagon}.

\medskip

If $A$ is a separable projective $R$-algebra,
then a right $A$-module $M$ is projective if and only if it is projective
as an $R$-module \cite[Proposition~2.14]{Saltman_1999_lectures_on_div_alg}.
This fact, together with Schanuel's lemma 
\cite[Corollary~5.6]{Lam_1999_lectures_on_modules_rings}, imply the following lemma.

\begin{lem}\label{LM:global-dim-of-Az}
	Let $A$ be a separable projective $R$-algebra
	and let $M\in \rMod{A}$.
	Then the projective dimension of $M$ as an $A$-module
	is at most the projective dimension of $M$ as an $R$-module.
	In particular, the right global dimension of $A$ is at most the global dimension of $R$.
\end{lem}

\begin{cor}\label{CR:bDer-iso}
	Let $A$ be a separable projective algebra over a regular ring $R$,
	and let $\bDer{\rmod{A}}$ be the bounded derived
	category of $\rmod{A}$. Then the natural
	functor $\bDer{A}:=\bDer{\rproj{A}}\to \bDer{\rmod{A}}$
	is an equivalence.
\end{cor}

\begin{proof}
	It is enough to show that all  finite 
	right $A$-modules have finite projective dimension
	(consult \cite[Tags~\href{https://stacks.math.columbia.edu/tag/0646}{0646},
	\href{https://stacks.math.columbia.edu/tag/064B}{064B}]{stacks_project}).
	Since $R$ is a regular, all finite $R$-modules have finite projective
	dimension \cite[Theorem~5.94]{Lam_1999_lectures_on_modules_rings}, and 
	Lemma~\ref{LM:global-dim-of-Az} completes the proof.
\end{proof}

\subsection{Homological Algebra}
\label{subsec:homological}

Let $A$ be a  right noetherian $R$-algebra. Write $\calK $ for the homotopy
category of chain complexes in $\rmod{A}$,
and $\calD =\bDer{\rmod{A}}$ for  the derived category of $\rmod{A}$.
We denote the shift-to-the-left-and-negate-the-differential functor $\calD\to \calD$ by
$T$, that is, given a chain complex
$P=(P_i,d_i)_{i\in\Z} = (\cdots\xrightarrow{d_3} P_2\xrightarrow{d_2} P_1\xrightarrow{d_1} P_0 \xrightarrow{d_0} \cdots)$,
we have $TP = (P_{i-1},-d_{i-1})_{i\in\Z}$.

As usual,
given $M\in\rmod{A}$, a  projective resolution of $M$
is a chain complex $P$ of objects in $\rproj{A}$ supported in non-negative degrees
together with a map $\alpha=\alpha_P:P_0\to M$ such that
$\cdots \to P_1\to P_0\to M\to 0$ is exact. We use $\alpha$ to
freely identify $\HH_0(P)$ with $M$.

Let $M,N\in \rmod{A}$ and $e\in\N\cup\{0\}$.
Following \cite[Tag \href{https://stacks.math.columbia.edu/tag/06XQ}{06XQ}]{stacks_project},
we define   $\Ext^e_A (M,N)$ as $\Hom_{\calD }(M,T^eN)$, where $M$ and $N$ are regarded
as chain complexes concentrated in degree $0$.
Note that this definition does not require choosing a projective resolution of $M$,
or an injective resolution of $N$.
Rather, 
if $P$ is a projective resolution of $M$,
then the map $\alpha_P:P_0\to M$ defines a quasi-isomorphism $P\to M$
in $\calD $, which gives rise to an isomorphism 
$\Ext^e_R(M,N)\to \Hom_{\calD }(P,T^eN) =\Hom_{\calK }(P,T^eN)$
(see \cite[Tags \href{https://stacks.math.columbia.edu/tag/05TG}{05TG}, \href{https://stacks.math.columbia.edu/tag/06XR}{06XR}]{stacks_project}).
It is straightforward to see that $\Hom_{\calK }(P,T^eN)$ is $-e$-th homology
group of the chain complex of $R$-modules
\[\Hom_A(P,N):=(\cdots \to \Hom_A(P_{-1},N)\to \Hom_A(P_0,N)\to \Hom_A(P_{1},N)\to\cdots)\]
with $\Hom_A(P_0,N)$ occurring in degree $0$. Thus, we have an $R$-module isomorphism
\[
u:\Ext^e_A (M,N)\to \HH_{-e}(\Hom_A(P,N)),
\]
which we usually suppress.
If $P'$ is another projective resolution of $M$, then there is a unique
morphism $f:P\to P'$ in $\calD $  such that 
$\HH_0(f)=\id_M$ \cite[Theorem~12.4]{Buhler_2010_exact_categories}. 
The morphism $f$ induces an isomorphism $\HH_{-e}( \Hom(f,N)):\HH_{-e}(\Hom_A(P',N))\to \HH_{-e}(\Hom_A( P,N))$
and one readily checks that the following diagram commutes:
\[
\xymatrix{
\Ext^e_A(M,N) \ar[r]^-u \ar[rd]_{u'} &
\HH_{-e}(\Hom_A(P,N))  \\
&
\HH_{-e}(\Hom_A(P',N)) \ar[u]_{\HH_{-e}( \Hom(f,N))}
}
\]

Likewise, if $M'\in\rmod{A}$, $\phi\in\Hom_A(M,M')$ and $P'$ is a projective resolution
of $M'$, then there is a unique $f:P\to P'$ with $\HH_0(f)=\phi$
and the induced morphism $\HH_{-e}(\Hom_A(f,N)):\HH_{-e}(\Hom_A(P',N))\to\HH_{-e}(\Hom_A(P,N))$
coincides
with $\Ext_A^e(f,N):\Ext_A^e(M',N)\to \Ext_A^e(M,N)$ upon identifying
$\Ext_A^e(M',N)$ and  $\Ext_A^e(M,N)$ with $\HH_{-e}(\Hom_A(P',N))$ and $\HH_{-e}(\Hom_A(P,N))$,
respectively.

\begin{lem}\label{LM:derived-cat-base-change}
	Let $A$ and $B$ be a noetherian $R$-algebras.
	Let $\calD$ be the bounded derived category of $\rmod{A}$ and let $X,Y\in \calD$.
	If $B$ is flat over $R$, then the base-change map $\Hom_{\calD}(X,Y)\otimes B\to
	\Hom_{\bDer{\rmod{A\otimes B}}}(X\otimes B, Y\otimes B)$ is an isomorphism.
\end{lem}

\begin{proof}
	We abbreviate $X\otimes B$ to $X_B$, $A\otimes B$ to $A_B$, and so on.
	Suppose that $X$ is concentrated in degrees $n,\dots,n+r$
	and $Y$ is concentrated in degrees $m,\dots,m+s$.
	We prove the lemma by induction on $r+s$. The case $r=s=0$
	is the famous fact that the natural map $\Ext^e_A(M,N)\otimes B\to \Ext^e_{A_B}(M_B,N_B)$
	is an isomorphism when $B$ is flat over $R$, and $M$ and $N$
	are right $M$-modules such that $N$ is finitely presented;
	see \cite[Theorem~2.39]{Reiner_2003_maximal_orders_reprint}, for instance.
	If $s>0$, let $K=\ker(d_n:X_n\to X_{n+1})$
	and let $X'=(0\to X_{n-1}/\im(d_n)\to X_{n-2}\to \cdots)$.
	Then we have a distinguished triangle $T^n K\to X\to X'\to T^{n+1}K$
	in $\calD$. Since $B$ is flat over $R$,
	$T^n K_B\to X_B\to X'_B\to T^{n+1}K_B$ is distinguished in 
	$ \bDer{\rmod{A_B}}$.
	These distinguished triangles give rise to a commutative diagram
	\[
	\xymatrixcolsep{0.8pc}\xymatrix{
	\cdots \ar[r]
	&
	\Hom(T^{n+1}K,Y)_B  \ar[r]\ar[d] &
	\Hom(X',Y)_B  \ar[r]\ar[d] & 
	\Hom(X,Y)_B	\ar[r]\ar[d]^{(*)} &
	\Hom(T^nX,Y)_B \ar[r]\ar[d] &
	\cdots
	\\
	\cdots \ar[r]
	&
	\Hom(T^{n+1}K_B,Y_B) \ar[r] &
	\Hom(X'_B,Y_B) \ar[r] & 
	\Hom(X_B,Y_B)	\ar[r]  &
	\Hom(T^nX_B,Y_B)  \ar[r]  &
	\cdots
	}
	\]
	in which the top and bottom row are exact \cite[\href{https://stacks.math.columbia.edu/tag/0149}{Tag 0149}]{stacks_project};
	the $\Hom$-groups are taken in $\calD$ or $\bDer{\rmod{A_B}}$. 
	Since $X'$ is concentrated in degrees
	$n+1,\dots,n+r$, the induction hypothesis and the Five Lemma
	imply that $(*)$ is an isomorphism. The case $s>0$ is handled similarly.
\end{proof}

\section{The Gersten--Witt Complex via Second-Residue Maps}
\label{sec:GW-second-res}

Let $R$ be a regular   ring,   
let $(A,\sigma)$ be
a separable projective $R$-algebra with involution,
and let $\veps\in \mu_2(R)$.
We now recall the definition of the (augmented) 
Gersten--Witt complex of $\veps$-hermitan over $(A,\sigma)$
constructed in \cite{Bayer_2019_Gersten_Witt_complex_prerprint} using generalized
second residue maps.
This complex is denoted $\aGW{A/R,\sigma,\veps}$, or just
$\aGW{A,\sigma,\veps}$, and takes the form
\[
0\to W_{\veps}(A,\sigma)
\xrightarrow{\dd_{-1}} \bigoplus_{\frakp\in R^{(0)}}\tilde{W}_\veps(A(\frakp) )
\xrightarrow{\dd_0} \bigoplus_{\frakp\in R^{(1)}}\tilde{W}_\veps(A(\frakp) )
\xrightarrow{\dd_1} \bigoplus_{\frakp\in R^{(2)}}\tilde{W}_\veps(A(\frakp) )
\to\cdots,
\]
where $\tilde{W}_\veps(A(\frakp) )$ stands
for $W_\veps(A(\frakp),\sigma(\frakp);\tilde{k}(\frakp))$ (notation as in Example~\ref{EX:line-bundle})
and 
\[
\tilde{k}(\frakp):=\Ext^{\hgt \frakp}_{R_\frakp}(k(\frakp),R_{\frakp})\cong
\Hom_R({\textstyle\bigwedge}_{R_\frakp}^{\hgt \frakp}(\frakp_\frakp/\frakp_\frakp^2),k(\frakp)) 
\]
(see \cite[Proposition~1.9(iii)]{Bayer_2019_Gersten_Witt_complex_prerprint}
regarding the  isomorphism).
Note that $\tilde{k}(\frakp)\cong k(\frakp)$ as $R$-modules, and thus
$\tilde{W}_\veps(A(\frakp))\cong W_{\veps}(A(\frakp),\sigma(\frakp))$, but this
isomorphism is not canonical unless $\hgt \frakp=0$.
We shall write $\GW{A,\sigma,\veps}$ for the complex obtained from $\aGW{A,\sigma,\veps}$
by removing the $-1$-term.

The $\frakp$-component of the differential $\dd_{-1}$ is given by localizing at $\frakp$.
For $e\geq 0$, the differential $\dd_e$ is 
defined to be $\sum_{\frakp,\frakq}\partial_{\frakp,\frakq}$,
where the sum ranges over all $\frakp\in R^{(e)}$ and $\frakq\in R^{(e+1)}$
with $\frakp\subseteq \frakq$, and $\partial_{\frakp,\frakq}$ is defined
as follows.

We base-change from $R$ to $R_\frakq$ in order to assume    
that $R$ is a regular local ring of dimension $e+1$
and $\frakq$ is its maximal ideal.
Set $S=R/\frakp$, $\frakm=\frakq/\frakp$, $\tilde{S}=\Ext^e (S,R)$ and $\tilde{\frakm}^{-1}=
\Ext^e(\frakm,R)$. As explained in \cite[\S2.1]{Bayer_2019_Gersten_Witt_complex_prerprint},
the localization-at-$\frakp$ maps
from $A\otimes \tilde{S}$ and $A\otimes \tilde{\frakm}^{-1}$ to $A\otimes \tilde{k}(\frakp)$
are injective, and when regarding
$A\otimes \tilde{S}$ and $A\otimes \tilde{\frakm}^{-1}$ as submodules of $A\otimes \tilde{k}(\frakp)$,
we have
\begin{equation}\label{EQ:Amtilde-property}
A\otimes\tilde{\frakm}^{-1} = \{x\in A\otimes\tilde{k}(\frakp)\where  x\frakm\subseteq
A\otimes \tilde{S}\}.
\end{equation}
Furthermore, we have $\Ext^{e+1}_R(S,R)=0$ and
$\Ext^e_R(k(\frakq),R)=0$, and therefore
a    short
exact sequence 
$ \tilde{S}\embeds  \tilde{\frakm}^{-1} \onto  \tilde{k}(\frakq)$,
induced by 
$\frakm\embeds S\onto k(\frakq)$. 
Since $A$
is flat over $R$, the induced sequence
\begin{equation}\label{EQ:fundamental-exact-seq}
0\to A\otimes \tilde{S}\to A\otimes\tilde{\frakm}^{-1} \xrightarrow{\calT} A\otimes \tilde{k}(\frakq) 
\to 0
\end{equation}
is also exact.
The map indicated by $\calT$ will be denoted as
$\calT_{\frakp,\frakq}$ or $\calT_{\frakp,\frakq,A}$ when
there is a risk of confusion.

Let $(V,f)\in \tHerm[\veps]{A(\frakp)}:=\Herm[\veps]{A(\frakp),\sigma(\frakp);\tilde{k}(\frakp)}$.
An \emph{$A$-lattice} in $V$ is a finitely generated $A\otimes S$-submodule $U$ of $V$
such that $U\cdot k(\frakp)=V$. Given an $A$-lattice $U$ in $V$, write
\[
U^f=\{x\in V\suchthat \hat{f}(U,x)\subseteq A\otimes \tilde{S}\}.
\]
Then $U^f$ is also an $A$-lattice in $V$ and $U^{ff}=U$ \cite[Lemma~2.1]{Bayer_2019_Gersten_Witt_complex_prerprint}.
Moreover, this source also tells us that there exists an $A$-lattice
$U$ in $V$ such that $U^f\frakm\subseteq U\subseteq U^f$.
Choosing an $A$-lattice $U$ in $V$
with $U^f\frakm\subseteq U\subseteq U^f$,
\eqref{EQ:Amtilde-property} allows us to define $\what{\partial f}: U^f/U\times U^f/U\to A\otimes \tilde{k}(\frakq)$
by
\[
\what{\partial f}(x+U,y+U)=\calT_{\frakp,\frakq}(f(x,y)).
\]
Write $(U^f/U,\partial f)$ for the corresponding hermitian space
in $\tHerm{A(\frakq)}$ and set
\[
\partial_{\frakp,\frakq}[V,f]=[U^f/U,\partial f].
\]
This is independent of the choice of $U$ \cite[Lemma~2.3]{Bayer_2019_Gersten_Witt_complex_prerprint}.
 
\section{The Balmer--Walter Construction of The Gersten--Witt Complex}
\label{sec:GW-Balmer}

Let $R,A,\sigma,\veps$ be as in Section~\ref{sec:GW-second-res}.
We proceed with recalling   the Balmer--Walter construction
of the Gersten--Witt complex of $R$, extending
it to $\veps$-hermitian forms over $(A,\sigma)$ in the process.
We also recall Gille's Gersten--Witt complex.

We refer the reader to \cite{Balmer_2000_Triangular_Witt_I}
and \cite{Balmer_2001_Triangular_Witt_II}
for all the necessary definitions concerning triangulated
hermitian categories (also called triangulated categories with duality). 
A concise
treatment is given in \cite[\S\S1--2]{Balmer_2002_Gersten_Witt_complex}.

\medskip

Let $(*,\omega)$ denote the hermitian structure induced by $\sigma$
on $\rproj{A}$ (Example~\ref{EX:line-bundle} with $M=R$).
As explained in \cite[\S2.6, \S2.8]{Balmer_2001_Triangular_Witt_II},
the bounded derived category
of $\rproj{A}$, denoted $\bDer{A}$, inherits
a triangulated hermitian structure, which we denote by $(*,1,\omega)$.
Specifically, if $P=(P_i,d_i)_{i\in \Z}=(\cdots\to P_1
\xrightarrow{d_1} P_0 \xrightarrow{d_0} P_{-1} \to \cdots)$,
then $P^*=(P_{-i}^*,d_{1-i}^*)_{i\in \Z}=(\cdots\to P_{-1}^*
\to P_0^* \to P_{ 1}^* \to \cdots)$, and $\omega_P=(\omega_{P_i})_{i\in \Z}$.
The corresponding $n$-th shifted hermitian structure (see \cite[p.~131]{Balmer_2002_Gersten_Witt_complex})
is
$(D_n,\delta_n,\omega_n):=(T^n\circ *,   (-1)^n,  (-1)^{n(n+1)/2} \omega)$,
and we write $W_\veps^n(A,\sigma)$ or $W_\veps^n(\bDer{A},*,1,\omega)$
or $W_\veps^n(\bDer{A})$
for the Witt group of $(\bDer{A},D_n,\delta_n,\veps\omega_n)$.

For every $e\geq 0 $, let $\calD_e=\bDer[e]{A}$ denote the full   subcategory 
of $\bDer{A}$ consisting of chain complexes   with homology $R$-supported in codimension $ e$.
Then   we have a filtration, 
\begin{align} \label{EQ:codim-filtration}
\bDer{A}&=\bDer[0]{A}\supseteq \bDer[1]{A}\supseteq 
\bDer[2]{A}\supseteq  \dots
\end{align}
in which
every term is a full subcategory of $\bDer{A}$   closed under shifts, mapping cones, isomorphisms, direct summands 
and $*$. (Note that $\bDer[e]{A}=0$ if $e>\dim R$.) 
Balmer \cite[Theorem~6.2]{Balmer_2000_Triangular_Witt_I}
(see also \cite[Theorem~7.1]{Balmer_2002_Gersten_Witt_complex})
showed the that the localization sequence  $\calD_{e+1}\to \calD_e\to \calD_e/\calD_{e+1}$ 
gives rise to a long exact sequence of 
Witt groups,
\begin{align*} 
\cdots&\to 
W^n_\veps(\calD_{e+1})\to
W^n_\veps(\calD_e)\to
W^n_\veps(\calD_e/\calD_{e+1})\xrightarrow{\partial_n}
W^{n+1}_\veps(\calD_{e+1})
\to\cdots,
\end{align*}
in which the left and middle maps are induced by the  evident functors,
and the right arrow is given by taking cones in the sense of 
\cite[Definitions~5.16, 2.10]{Balmer_2000_Triangular_Witt_I}. 
Writing $\dd_{-1}$ for the natural map
$W^0_\veps(\calD_0)\to W^0_\veps(\calD_0/\calD_1)$
and $\dd_e$ ($e\geq 0$) for the composition
$
W^e_\veps(\calD_e/\calD_{e+1})\xrightarrow{\partial_e}
W^{e+1}_\veps(\calD_{e+1})\to
W^{e+1}_\veps(\calD_{e+1}/\calD_{e+2}) 
$,
we get a cochain complex
\begin{align*} 
&0\to
W^0_\veps(\calD_0)\xrightarrow{\dd_{-1}}
W^0_\veps(\calD_0/\calD_1)\xrightarrow{\dd_0}
W^1_\veps(\calD_1/\calD_2)\xrightarrow{\dd_1}
W^2_\veps(\calD_2/\calD_3)\xrightarrow{\dd_2}
\cdots
,
\end{align*}
which we denote by $\iaGW{A/R,\sigma,\veps}$, 
or simply $\iaGW{A,\sigma,\veps}$. 
This is the (augmented) Gersten--Witt complex of $(A,\sigma,\veps)$ \`a la Balmer and Walter
\cite{Balmer_2002_Gersten_Witt_complex}.
We write $\iGW{A,\sigma,\veps}$ for the complex obtained from $\iaGW{A,\sigma,\veps}$
by removing the $-1$-term.

In the case
$(A,\sigma,\veps)=(R,\id_R,1)$,
Balmer and Walter \cite[\S\S6--7]{Balmer_2002_Gersten_Witt_complex}
defined a canonical isomorphism 
--- called \emph{d\'evissage} ---
between
the respective terms of $\iaGW{A,\sigma,\veps}$ and $\aGW{A,\sigma,\veps}$.
Bayer-Fluckiger, Parimala and the author
\cite[Proposition~2.7]{Bayer_2019_orders_locally_iso} showed that this isomorphism also 
respects the differentials.
Here we shall generalize this to all $A,\sigma,\veps$:

\begin{thm}\label{TH:GW-well-defined}
	There is a canonical isomorphism
	$s:\iaGW{A,\sigma,\veps}\to \aGW{A,\sigma,\veps}$.
\end{thm}

\begin{remark}
Gille's Gersten--Witt complex
\cite{Gille_2007_hermitian_GW_complex_I}, \cite{Gille_2009_hermitian_GW_complex_II}
is defined when $\dim R$ is finite, and 
is isomorphic to $\iaGW{A,\sigma,\veps}$. This is explained in 
\cite[\S2.10, Example~4.4]{Gille_2007_hermitian_GW_complex_I}
and \cite[pp.~349--350]{Gille_2009_hermitian_GW_complex_II}. Briefly,
let 
$\bDer[c]{\rMod{A}}$ denote the full subcategory
of $\bDer{\rMod{A}}$ consisting of chain complexes with finitely generated homology.
Fix an injective resolution $I_\bullet\in \bDer[c]{\rMod{R}}$  of $R$
and set $J_\bullet =A\otimes_R I_\bullet$. 
For every $F_\bullet \in \bDer[c]{\rMod{A}}$, let $F_\bullet^\#$ denote
the $\Hom$-chain complex ${\calH om}_A(F_\bullet,I_\bullet)$, regarded as complex of \emph{right} $A$-modules
by twisting via $\sigma$. There is a natural isomorphism $\eta: F\to F^{\#\#}$
making $(\bDer[c]{\rMod{A}},\#,1,\eta)$ into a triangularted hermitian category,
see \cite[\S3.7]{Gille_2009_hermitian_GW_complex_II}. Moreover,
the isomorphism $A_A\cong J_\bullet$ in $ \bDer[c]{\rMod{A}}$
induces a   natural isomorphism   $i:P_\bullet^*\to P_\bullet^\#$ for all $P_\bullet\in \bDer{A}=\bDer{\rproj{A}}$
such that $(\id,i):(\bDer{A},*,1,\omega)\to(\bDer[c]{\rMod{A}},\#,1,\eta)$ is
a $1$-hermitian $1$-exact functor (in the sense of \cite[\S4]{Balmer_2002_Gersten_Witt_complex}).
Gille's Gersten Witt complex is defined exactly as $\iaGW{A,\sigma,\veps}$
with the difference that $(\bDer{A},*,1,\omega)$ is 
replaced with $(\bDer[c]{\rMod{A}},\#,1,\eta)$.
By Corollary~\ref{CR:bDer-iso},
$(\id,i)$ is  an equivalence of triangulated hermitian categories.
This equivalence   respects the codimension filtration
\eqref{EQ:codim-filtration} on both $\bDer{A}$ and $\bDer[c]{\rMod{A}}$, 
so Gille's Gersten--Witt complex is isomorphic
to $\iaGW{A,\sigma,\veps}$.

Gille also showed that the respective terms of
$\iaGW{A,\sigma,\veps}$ and $\aGW{A,\sigma,\veps}$ are isomorphic for all $A,\sigma,\veps$,
but it is not clear that  this isomorphism 
  respects  the differentials. 
\end{remark}

Theorem~\ref{TH:GW-well-defined} affords
a simple proof of  the following proposition, which appears
as Theorem~2.9 in \cite{Bayer_2019_Gersten_Witt_complex_prerprint}.

\begin{prp}\label{PR:base-of-GW-does-not-matter}
	Let $R'$ be another regular ring and suppose that
	$A$ is equipped with an $R'$-algebra structure
	such  
	that $(A,\sigma)$ is a separable projective
	$R'$-algebra with involution.
	Then $\aGW{A/R,\sigma,\veps}\cong \aGW{A/R',\sigma,\veps}$.
\end{prp} 

\begin{proof}
	Let $R_1=\Cent(A)^{\{\sigma\}}$.
	Recall from~\ref{subsec:Azumaya}
	that $(A,\sigma)$ is Azumaya over $R_1$ and $R_1$ is 
	finite \'etale over $R$, hence regular. Since the structure
	morphism $R'\to \Cent(A)$ factors via $R_1$, it is enough
	to prove the proposition when $R'=R_1$.
	Now, 
	\emph{going down} and \emph{incomparability} for prime ideals
	\cite[Tags \href{https://stacks.math.columbia.edu/tag/00HS}{00HS},
	\href{https://stacks.math.columbia.edu/tag/00GT}{00GT}]{stacks_project}
	imply that an $R'$-module is $R'$-supported in codimension
	$e$ if and only if it is $R$-supported in codimension $e$.
	Consequently,	
	the filtration \eqref{EQ:codim-filtration} of $\bDer{A}$  is the same for $R$ and $R'$,
	so $\iaGW{A/R,\sigma,\veps}\cong
	\iaGW{A/R',\sigma,\veps}$.
	Theorem~\ref{TH:GW-well-defined} completes the proof.
\end{proof}

The proof of Theorem~\ref{TH:GW-well-defined} is given  in the next three sections:
Section~\ref{sec:transfer}  establishes a preliminary result about transfer in hermitian categories.
This is used in the construction of   $s=(s_n)_{n\geq -1}$, 
which   is given in Section~\ref{sec:GW-propositions}.
Finally,  Theorem~\ref{TH:GW-well-defined} is proved in Section~\ref{sec:proof-of-dd-desc}.

\section{Transfer in Hermitian Categories}
\label{sec:transfer}

We introduce a variation of \emph{transfer into the endomorphism ring} in hermitian categories,
which will play a key role in the definition of the isomorphism
$s$ of Theorem~\ref{TH:GW-well-defined}.
To that end, we first recall hermitian forms valued in bimodules with involution in
the sense of \cite{First_2015_general_bilinear_forms}. 
	
	Let $A$ be an $R$-algebra 
	(no involution on $A$ is assumed),
	and let $Z$ be an $(A^\op,A)$-progenerator,
	i.e., an $(A^\op,A)$-bimodule such that $Z_A$ is finite projective, $A_A$ is a summand of $Z^n$
	for some $n\in \N$, and $A^\op=\End_A(Z )$; see \cite[\S18B]{Lam_1999_lectures_on_modules_rings}
	for further details.
	Suppose further that  $r z=z r$ for all $r\in R$, $z \in Z$, 
	and let $\theta:Z\to Z$ be an $R$-automorphism
	(written exponentially)
	satisfying $(a^\op z b)^\theta=b^\op z^\theta a$ and $z^{\theta\theta}=z$
	for all $a,b\in A$, $z\in Z$. 
	Finally, let $\veps\in \mu_2(R)$.

	Following \cite[\S2]{First_2015_general_bilinear_forms},
	define a  \emph{$Z$-valued $\veps\theta$-hermitian} space\footnote{ 
		In  \cite{First_2015_general_bilinear_forms}, the name
		is ``(general) $\veps\theta$-symmetric bilinear space''.
	} over $A$ to be a pair $(P,\hat{f})$
	consisting of $P\in \rproj{A}$
	and a   biadditive map $\hat{f}:P\times P\to Z$
	satisfying $\hat{f}(xa,x'a')=a^\op \hat{f}(x,x')a'$
	and $\hat{f}(x,x')=\veps \hat{f}(x',x)^\theta$
	for all $a,a'\in A$, $x,x'\in P$.

	The pair $(Z,\theta)$ induces a hermitian structure $(*,\omega)$
	on $\rproj{A}$  such
	that  $\veps$-hermitian forms over $ \rproj{A} $ correspond to 
	$Z$-valued $\veps\theta$-hermitian forms: Given an object $P$
	and a morphism $\vphi$ in $\rproj{A}$,
	define
	$P^*$ to be $\Hom_A(P,Z )$ endowed with the right
	$A$-module structure given by $(\phi a)x=a^\op(\phi x)$ ($\phi\in P^*$, $a\in A$, $x\in P$),
	let $\vphi^*=\Hom_A(\vphi,Z )$
	and define $\omega_P:P\to P^{**}$ by $(\omega_P x)\phi=(\phi x)^\theta$ ($x\in P$, $\phi\in P^*$).
	Notice that $P^*\in \rproj{A}$ 
	and $\omega_P$ is an isomorphism because $Z$ is an $(A^\op,A)$-progenerator,
	see \cite[Lemmas~4.2, 4.4]{First_2015_morita_equiv_to_op}.
	If $(P,f)$ is an $\veps$-hermitian space over $(\rproj{A},*,\omega)$,
	then its corresponding $Z$-valued $\veps\theta$-hermitian form over $A$
	is $\hat{f}:P\times P\to Z$ given by $\hat{f}(x,y)=(fx)y$.

	\begin{example}
	Let $(A,\sigma)$ be an $R$-algebra with involution and let $M$ be an invertible $R$-module.
	Take    $Z= A\otimes M $, $\theta= \sigma\otimes \id_M $,
	and make $Z$ into an $(A^\op,A)$-bimodule
	by setting $a^\op\cdot z \cdot b:=a^\sigma z b$ ($a,b\in A$, $z \in Z$).
	Then $Z$-valued $\theta$-hermitian forms over $A$ are the same thing
	as $(A\otimes M,\sigma\otimes\id_M)$-valued $1$-hermitian forms over $(A,\sigma)$
	in the sense of Example~\ref{EX:line-bundle}.
	\end{example}

 	Let $(\catC,*,\omega)$ be any idempotent complete 
	(see \cite[\S6]{Buhler_2010_exact_categories}) hermitian $R$-category
	and let $P_0$
	be an object such that every object of $\catC$
	is a summand of $P_0^n$ for some $n\in\N$.
	By \cite[Lemmas~II.3.2.3, II.3.3.2]{Knus_1991_quadratic_hermitian_forms} (for instance),
	the functor 
	\begin{align*} 
	F=F_{P_0}:=\Hom_{\catC}(P_0,-) 
	\end{align*} 
	defines an equivalence from $\catC$ to $\rproj{E}$,
	where  $E:=\End_{\catC}(P_0)$. 
	Write $Z=F(P_0^*)=\Hom_\catC (P_0, P_0^*)$ and let $\theta:Z\to Z$
	be given by $\vphi^\theta =\vphi^*\circ \omega_{P_0}$.
	We make  $Z$ into an $(E^\op, E)$-bimodule by setting 
	$\alpha^\op \cdot \vphi \cdot \beta=\alpha^*\circ \vphi \circ \beta$.
	One readily checks that $\vphi^{\theta\theta}=\vphi$
	and $(\alpha^\op \vphi \beta)^\theta=\beta^\op \vphi^\theta \alpha$
	for all $\alpha,\beta\in E$, $\vphi\in Z$.
	Moreover, $Z$ is an   $(E^\op,E)$-progenerator.
	Indeed, $F  *:E^\op=\End_{\catC}(P_0)^\op\to \End_{E}(Z)$
	is an isomorphism, hence $\End_{E}(Z)=E^\op$,
	and since $P_0$ is isomorphic to  a summand of
	$(P_0^*)^n$ for some $n$, the right $E$-module $E =FP_0$
	is isomorphic to a summand of $Z^n=F(P_0^*)^n$ for the same $n$.
	
	Let 
	$(\tilde{*},\tilde{\omega})$ denote
	the hermitian structure on $\rproj{E}$ induced
	by $Z$ and $\theta$.
	There is an natural isomorphism $i:F*\to \tilde{*}F$ given  by
	\begin{align*}
	i_P: \Hom_\catC(P_0,P^*)&\to \Hom_E(\Hom_\catC(P_0,P),Z)\\ 
	\vphi&\mapsto 
	[\psi\mapsto \vphi^*\circ \omega_P\circ \psi]\nonumber
	\end{align*}
	(this is an isomorphism because
	$i_P=F\circ \Hom(\omega_P,P_0^*)\circ *$).
	It is straightforward to check that 
	\[(F,i): (\catC,*,\omega)\to (\calP(E),\tilde{*},\tilde{\omega})\]
	is a $1$-hermitian equivalence (see \ref{subsec:herm-cat}).
	We  call it the  
	\emph{$P_0$-transfer}. 
	One readily checks that if $(P,f)\in\Herm[\veps]{\catC}$ and $F(P,f)=(FP,f_1)$, then the induced
	$Z$-valued $\veps\theta$-hermitian form $\hat{f}_1:\Hom_{\catC}(P_0,P)\times \Hom_{\catC}(P_0,P)\to  \Hom_{\catC}(P_0,P_0^*)$
	is given by
	\begin{align*}
	\hat{f}_1(\vphi,\psi)= \vphi^*\circ \veps  f\circ \psi.
	\end{align*}

	\begin{remark}
		Ordinary transfer into the endomorphism
		ring, see \cite[Proposition~2.4]{Quebbemann_1979_hermitian_categories},
		requires one to specify  a unimodular $\gamma$-hermitian form $h_0$ on the object $P_0$.
		Given such a form,
		we can define an $R$-involution 
		$\sigma:E\to E$ by
		$\vphi^\sigma=  h_0^{-1}\vphi ^* h_0$
		and identify
		$Z=F(P_0^*)$ with 
		$E=FP_0$
		via 
		$Fh_0^{-1}$.
		Under this identification,
		$Z$-valued $\veps\theta$-hermitian forms
		become $ \gamma \veps$-hermitian forms over $(E,\sigma)$,
		and we get a $\gamma$-hermitian 
		equivalence $(\catC,*,\omega)\to (\rproj{E},*_{\sigma,R},\omega_{\sigma,R})$
		(notation as in Example~\ref{EX:line-bundle}),
		which  is precisely the transfer of  \cite[Proposition~2.4]{Quebbemann_1979_hermitian_categories}.
\end{remark}

\section{D\'evissage}
\label{sec:GW-propositions}

In this section, we
construct
the isomorphism $s=(s_e)_{e\geq -1}$
of Theorem~\ref{TH:GW-well-defined}.
We use the same notation as in Section~\ref{sec:GW-Balmer}.

When $e=-1$, we define $s_{-1}:W_\veps(A,\sigma)\to W^0_\veps(\bDer{A},*,1,\omega)$ 
to be the embedding-in-degree-$0$ homomorphism, 
which is an isomorphism by a theorem of Balmer \cite[Theorem~4.3]{Balmer_2001_Triangular_Witt_II}.
The case   $e\geq 0$ will occupy the rest of
this  section and  be concluded in Construction~\ref{CN:GW-isomorphism}.

\medskip

We begin by extending two results of Balmer and Walter \cite[\S\S6--7]{Balmer_2002_Gersten_Witt_complex}
from   $A=R$ to general $A$.

\begin{prp}\label{PR:derived-localization}
	There is an equivalence of triangulated hermitian  $R$-categories,
	\[
	\loc:\bDer[e]{A}/\bDer[e+1]{A}\xrightarrow{\sim} \coprod_{\frakp\in R^{(e)}}
	\bDer[e]{A_\frakp},
	\]
	the $\frakp$-component of which is given by localizing at $\frakp$.
	The $\frakp$-component of the implicit natural isomorphism  $\loc *\to *\loc$
	is the canonical isomorphism $(P^*)_\frakp\to (P_\frakp)^*$.
\end{prp}

\begin{proof}
	It is routine to check that $\mathrm{loc}$ is a $1$-exact $1$-hermitian functor
	in the sense of \cite[\S4]{Balmer_2002_Gersten_Witt_complex}. 
	It remains  to check that
	$\loc$ is an equivalence of triangulated categories.
	Since the natural functor $\bDer{A}\to \bDer{\rmod{A}}$
	is an equivalence (Corollary~\ref{CR:bDer-iso}),
	this can be shown as in the proof of
	\cite[Proposition~7.1]{Balmer_2002_Gersten_Witt_complex}
	(see also the proof of \cite[Theorem~5.2]{Gille_2007_graded_GW_complex}). 
\end{proof}

Fix $\frakp\in R^{(e)}$ for the remainder
of the discussion.
Abusing the notation, we denote
the triangulated hermitian structure that $\sigma$
induces on $\bDer[e]{A_\frakp}$ by $(*,1,\omega)$
and write the corresponding shifted
hermitian structures
as $(D_n,\delta_n,\omega_n)_{n\in\Z}$.
We further define 
\[\calC(A_\frakp)\] to be the full subcategory of $\bDer[e]{A_\frakp}$
consisting of chain complexes $P $ supported in degrees
$0,\dots,e$
and satisfying  $\HH_i(P )=0$ for all $i\neq   0$.
Then $D_e$ restricts to a duality  on $\calC(A_\frakp)$,
because $\Ext^i_{A_\frakp}(M,A_\frakp)=0$ for any finite-length
$A_\frakp$-module $M$
and all $i\neq e$ \cite[Proposition~1.7(i)]{Bayer_2019_Gersten_Witt_complex_prerprint}.
Moreover, $\calC(A_\frakp)$ is abelian.
Indeed, $A_\frakp$ has global dimension at most $e=\dim R_\frakp$,
and so $\HH_0$
defines an equivalence from $\calC^0(A_\frakp)$ to $\fl{A_\frakp}$, the category
of finite-length $A_\frakp$-modules.
Thus, $(\calC(A_\frakp),D_e,\omega_e)$ is an exact hermitian category,
and   we may consider its Witt group.

\begin{prp}\label{PR:GW-reduction-I}
	There is
	an
	isomorphism 
	\[W_\veps(\calC(A_\frakp))\xrightarrow{\sim} W_\veps^0(\bDer[e]{A_\frakp},D_e,\delta_e,\omega_e)=
	W^e_\veps(\bDer[e]{A_\frakp})\]
	defined by regarding a Witt class in $W_\veps(\calC(A_\frakp))$
	as a Witt class in $W^e_\veps(\bDer[e]{A_\frakp})$.
\end{prp}

\begin{proof}
	The proof is in the spirit 
	of the proof of \cite[Lemma~6.4]{Balmer_2002_Gersten_Witt_complex}.

	Write $\calC:=\calC(A_\frakp)$.
	The hermitian structure $(D_e,\omega_e)$ on $\calC$
	induces triangulated hermitian structures
	on   $\bDer{\calC }$ and $\bHot{\calC}$ (the bounded homotopy category
	of $\calC$),
	both denoted  $(\tilde{D} ,1,\tilde{\omega} )$.  
	We represent objects of $\bDer{\calC }$ and $\bHot{\calC}$
	as bounded double chain complexes.
	Explicitly, the double complex $P_{\bullet,\bullet}=(P_{i,j},h_{i,j},v_{i,j})_{i,j\in \Z}$
	($h_{i,j}$ are the horizontal differentials and $v_{i,j}$
	are the vertical differentials) corresponds to
	$(\dots\to P_{1,\bullet}\to P_{0,\bullet}\to P_{-1,\bullet}\to\dots)$
	in $\bDer{\calC}$ or $\bHot{\calC}$. 
	Then
	$\tilde{D} P_{\bullet,\bullet} =(P^*_{-i,e-j},h_{1-i,e-j}^*,(-1)^e v_{-i,e+1-j}^*)_{i,j}$
	and $\tilde{\omega}_{P_{\bullet,\bullet}}=((-1)^{\frac{1}{2}e(e+1)}\omega_{P_{i,j}})_{i,j}$.

	It is routine to check that the total complex construction 
	\[\Tot: P_{\bullet,\bullet}\mapsto 
	\Big(\bigoplus_{i\in\Z} P_{i,n-i},\sum_i (h_{i,n-i}+(-1)^iv_{i,n-i})\Big)_{n\in\Z}\]
	defines a functor from $\bHot{\calC}$ to $\bDer[e]{A_\frakp}$.
	We claim that $\Tot$   extends via localization to $\bDer{\calC}$.
	To see this,
	let $H:\bDer{\calC}\to\bDer{\fl{A_\frakp}}$ denote the 
	exact functor induced by the equivalence $\HH_0:\calC\to \fl{A_\frakp}$, namely,
	$HP_{\bullet,\bullet}=(\HH_0(P_{i,\bullet}),\HH_0(h_{i,\bullet}))_{i\in \Z}$.
	Consider the natural transformation $t_P:\Tot P_{\bullet,\bullet}\to HP_{\bullet,\bullet}$
	given as the composition $(\Tot P_{\bullet,\bullet})_i\to P_{i,0}\to P_{i,0}/\im v_{i,1}=\HH_0(P_{i,\bullet})$
	in degree $i$.
	Since $\HH_j(P_{i,\bullet})=0$ for $j>0$,
	standard diagram chasing (e.g., as in the
	proof of \cite[\href{https://stacks.math.columbia.edu/tag/0E1R}{Tag 0E1R}]{stacks_project}) 
	shows that  $\HH_i(t_P):\HH_i(\Tot P_{\bullet,\bullet})\to \HH_i(HP_{\bullet,\bullet})$
	is an isomorphism.
	Thus, 
	$t_P:\Tot P_{\bullet,\bullet}\to HP_{\bullet,\bullet}$ is
	a natural quasi-isomorphism   in
	$\bHot{\rmod{A_\frakp}}$.
	Now, if $f :P_{\bullet,\bullet}\to P'_{\bullet,\bullet}$
	is a quasi-isomorphism in $\bHot{\calC}$,
	then $\HH_i (H f ):\HH_i( H P_{\bullet,\bullet})\to \HH_i (H P'_{\bullet,\bullet})$
	is an isomorphism
	in $\fl{A_\frakp}$,
	and thus, so is 
	$\HH_i(\Tot f ):\HH_i(\Tot P_{\bullet,\bullet})\to \HH_i(\Tot P'_{\bullet,\bullet})$. 
	This means that $\Tot f $ is a quasi-isomorphism,
	so  
	$\Tot$ extends via localization to a functor $\bDer{\calC }\to \bDer[e]{A_\frakp}$.

	Define 
	$i:\Tot   \tilde{D}\to D_e  \Tot$
	by 
	$i_{P_{\bullet,\bullet}}= (\bigoplus_{i\in\Z} (-1)^{ei}\id_{P_{i,e-n-i}^*} )_{n\in\Z}$.
	It is routine to check that 
	$i_{\tilde{D}P_{\bullet,\bullet}}\circ \Tot \tilde{\omega}_{P_{\bullet,\bullet}}=D_e i_{P_{\bullet,\bullet}}
	\circ \omega_{e,\Tot P_{\bullet,\bullet}}$
	and $(-1)^e Ti_{P_{\bullet,\bullet}}=   i_{T^{-1}P_{\bullet,\bullet}}$.
	Since $\Tot$ commutes with $T$ and preserves cones, 
	we conclude that
	\[
	(\Tot,i):(\bDer{\calC },\tilde{D},1,\tilde{\omega}) \to 
	(\bDer[e]{A_\frakp},D_e,\delta_e,\omega_e) .
	\]
	is   a $1$-exact $1$-hermitian functor in the sense of \cite[\S4]{Balmer_2002_Gersten_Witt_complex}.
	
	We claim that $\Tot: \bDer{\calC }
	\to \bDer[e]{A_\frakp}$ is an equivalence of triangulated
	categories.
	To see this, fix an inverse functor
	$F$ to $H:\bDer{\calC}\to \bDer{\fl{A_\frakp}}$,
	and let $G:\bDer[e]{A_\frakp}\to \bDer[e]{\rmod{A_\frakp}}$
	be the functor induced by the inclusion
	$\rproj{A_\frakp}\to \rmod{A_\frakp}$.
	Then $G$ is an equivalence
	by Corollary~\ref{CR:bDer-iso}.
	Given $M_\bullet\in \bDer{\fl{A}}$
	with $P_{\bullet,\bullet}:=FM_{\bullet}$,
	we observed above that there is a natural  isomorphism
	$\Tot FM_\bullet =\Tot P_{\bullet,\bullet}\to H P_{\bullet,\bullet}\cong M_\bullet$
	in $\bDer[e]{\rmod{A_\frakp}}$.
	Thus, $G\Tot F$ is equivalent to the natural
	functor $\bDer{\fl{A_\frakp}}\to \bDer[e]{\rmod{A_\frakp}}$.
	The latter is an equivalence of triangulated categories by 
	\cite[\S1.15, Lemma, Example~(b)]{Keller_1999_cyclic_homology},
	so  $\Tot$ is an equivalence as well.
	
 	To conclude, $(\Tot,i)$ 
 	induces an isomorphism
 	from $
 	W_\veps^0(\bDer{\calC },\tilde{D},1,\tilde{\omega})$ to 
 	$W_\veps^0(\bDer[e]{A_\frakp},D_e,\delta_e,\omega_e)
 	$.
 	In addition, by \cite[Theorem~4.3]{Balmer_2001_Triangular_Witt_II},
 	em\-bedding-\-in-\-deg\-ree-$0$ induces
 	an isomorphism
 	$W_\veps( \calC  ,D_e ,\omega_e)\to 
 	W_\veps^0(\bDer{\calC },\tilde{D},1,\tilde{\omega})$.
 	The composition of these two maps  is the map considered
 	in the proposition,  
	so we are done. 	
\end{proof}

Let $\calC^0(A_\frakp)$ denote the full subcategory
of semisimple objects in $\calC(A_\frakp)$;
it is   closed under $D_e$. Since the Jacobson radical of $A_\frakp$
is $A_\frakp   \frakp$ \cite[Lemma~1.5]{First_2022_octagon}, 
and since $A(\frakp)\cong A_\frakp/A_\frakp\frakp$ is semisimple artinian,
the full subcategory of semisimple objects in $\fl{A_\frakp}$ is
$\rproj{A(\frakp)}$.
As a result,
the equivalence $\HH_0:\calC(A_\frakp)\to \fl{A_\frakp}$
restricts to an equivalence
$\HH_0:\calC^0(A_\frakp)\to\rproj{A(\frakp)}$.

\begin{prp}\label{PR:GW-reduction-II}
	The inclusion $1$-hermitian functor $\calC^0(A_\frakp)\to \calC(A_\frakp)$
	induces an isomorphism
	$W_\veps(\calC^0(A_\frakp),D_e,\omega_e)\to W_\veps(\calC(A_\frakp),D_e,\omega_e)$.
\end{prp}

\begin{proof}
	This is a special case of   a 
	theorem of Quebbemann, Scharlau and Schulte \cite[Corollary~6.9, Theorem~6.10]{Quebbemann_1979_hermitian_categories}.
\end{proof}

We proceed with showing that $W_\veps(\calC^0(A_\frakp),D_e,\omega_e)$
is canonically isomorphic to $\tilde{W}(A(\frakp))=W_\veps(A(\frakp),\sigma(\frakp);\tilde{k}(\frakp))$, which will
finish the construction of $s_e$.
This was shown by Balmer and Walter \cite[Theorem~6.1]{Balmer_2002_Gersten_Witt_complex} in the case $A=R$,
but their proof does not extend to our more general situation.
We therefore take a different approach, which is based on 
the transfer hermitian functor of Section~\ref{sec:transfer}.
To that end, we introduce additional notation.

Let $K$  
denote  a minimal resolution of the $R_\frakp$-module $k(\frakp)$;
recall from \ref{subsec:homological} 
that $K=(K_i,d_i)_{i\in \Z}$ comes equipped with a morphism
$\alpha:K_0\to k(\frakp)$ which we use implicitly to identify
$\HH_0(K)$ with $k(\frakp)$.
The actual choice of $K$ will not matter in the end as long as $K$ is supported
in degrees $0,\dots,e$.
Write $K_A=A\otimes K$ and note that $\HH_0(K_A)=A(\frakp)$ and $K_A\in\calC^0(A_\frakp)$ by 
the flatness of $A$ over $R$.
We identify 
$\End_{\bDer{A_\frakp}}(K_A)$ with $A(\frakp)$  
via the isomorphism
\[
\End_{\calC(A_\frakp)}(K_A)\xrightarrow[\HH_0]{\sim} \End_A(A(\frakp))
\xrightarrow[\phi\mapsto \phi(1)]{\sim} A(\frakp).
\]
Now, as explained in Section~\ref{sec:transfer}, $ Z(K):=\Hom_{\bDer{A_\frakp}}(K_A,D_eK_A)$
is naturally an $(A(\frakp)^\op,A(\frakp))$-bi\-mo\-dule. Explicitly,
given $\overline{a},\overline{b}\in A(\frakp)$ which lift to
$a,b\in A$ and 
$\vphi\in Z(K)$,
we have $\quo{a}^\op \cdot \vphi\cdot \quo{b} = D_e(\ell_a\otimes \id_K) \circ\vphi\circ (\ell_b \otimes \id_K)$,
where $\ell_a :A_\frakp\to A_\frakp$ is left-multiplication by $a$.

\begin{prp}\label{PR:Koszul-transfer}
	Let $(*(K),\omega(K))$ denote
	the hermitian structure on $\rproj{A(\frakp)}$
	induced by the $(A(\frakp)^\op,A(\frakp))$-bi\-mo\-dule
	$Z(K)=\Hom_{\bDer{A_\frakp}}(K_A,D_eK_A)$ and
	the map $ \theta(K):Z(K)\to Z(K)$ given by $\vphi^{\theta(K)} = D_e \vphi \circ \omega_{e,K_A}$.
	Then $K_A$-transfer (see Section~\ref{sec:transfer})
	induces an exact $1$-hermitian equivalence 
	\[
	(F,j)=(F^{(K)},j^{(K)}):(\calC^0(A_\frakp),D_e,\omega_e)\to (\rproj{A(\frakp)},*(K),\omega(K)).
	\]
\end{prp}

\begin{proof}
	This will follow from Section~\ref{sec:transfer}
	once we show that   every   object in   $\calC^0(A_\frakp)$ is a summand of
	$K_A^n$ for some $n\in\N$, and that $F$ is exact.
	The former is a consequence of the equivalence
	$\HH_0:\calC^0(A_\frakp)\to \rproj{A(\frakp)}$
	and the fact that every  module in $\rproj{A(\frakp)}$
	is a summand of $A(\frakp)^n$ for some $n\in\N$.  
	The exactness of $F$ is automatic
	because the categories $\calC^0(A_\frakp)$ and $\rproj{A(\frakp)}$
	are abelian and semisimple. 
\end{proof}

Now, in order to show that $W_\veps(\calC^0(A_\frakp),D_e,\omega_e)\cong \tilde{W}_\veps(A(\frakp))$,
it is enough to identify 
$(Z(K),\theta(K))$ with $(A\otimes \tilde{k}(\frakp),\sigma\otimes \id_{\tilde{k}(\frakp)})$. 
In the end, it will turn out that the resulting isomorphism
$W_\veps(\calC^0(A_\frakp))\to \tilde{W}_\veps(A_\frakp)$
is independent of $K$.

\medskip

We begin by applying the construction of $\calC^0(A_\frakp)$
with $(A,\sigma)=(R,\id_R)$ to form $\calC^0(R_\frakp)$,
which is equivalent to $\rproj{k(\frakp)}$ via $\HH_0$.
The   hermitian structure  of $\rproj{R_\frakp}$ is
denoted $(\vee,\zeta)$ and the induced
shifted hermitian structures
of  $\bDer{R_\frakp}$
are denoted $(\Delta_n,\delta_n,\zeta_n)_{n\in\Z}$. 
Base-changing along the structure homomorphism $R_\frakp\to A_\frakp$
induces a $1$-hermitian functor $(G,t):(\rproj{R_\frakp},\vee,\zeta)\to
(\rproj{A_\frakp},*,\omega)$. Explicitly,
for all $P\in \rproj{R_\frakp}$,
we have 
$GP=A\otimes P$ and 
$t_P:A\otimes P^\vee  \to (A\otimes P)^*$
is determined by
$(t_P(a\otimes \phi ))(a'\otimes x )=a^\sigma \cdot  (\phi x)\cdot   a'$
($\phi\in P^*$, $a,a'\in A$, $x\in P$).
This, in turn, induces a $1$-exact
$1$-hermitian functor $\bDer{R_\frakp}\to \bDer{A_\frakp}$,
which is also denoted $(G,t)$.
For $Q=(Q_i,d_i)_{i\in\Z}\in \bDer{R_\frakp}$, we write $GQ=A\otimes Q:=
(A\otimes Q_i,\id_A\otimes d_i)_{i\in \Z}$
as $Q_A$.

\begin{lem}\label{LM:Phi-P-Q-dfn}
	For every $P,Q\in \bDer{R_\frakp}$,
	define a natural transformation
	\[
	\Phi_{P,Q}:A\otimes \Hom_{\bDer{R_\frakp}}(P,\Delta_eQ) \to 
	\Hom_{\bDer{A_\frakp}}(P_A,D_eQ_A)
	\]
	by $\Phi_{P,Q}(a\otimes  u)=t_Q\circ (\ell_a \otimes u)$,
	where $\ell_a$ denotes left-multiplication   by $a$.
	Then:
	\begin{enumerate}[label=(\roman*)]
		\item 
		$\Phi_{P,Q}$ is an isomorphism.
		\item 
		$\Phi_{K,K}:A\otimes \Hom_{\bDer{R_\frakp}}(K,\Delta_e K)\to 
		\Hom_{\bDer{A_\frakp}}(K_A,D_eK_A)=Z(K)$
		is an isomorphism
		of $(A^\op,A)$-modules under which
		$\sigma\otimes \id $
		corresponds to
		$\theta(K)$ (see Proposition~\ref{PR:Koszul-transfer}).
		Here, we regard  $A\otimes \Hom_{\bDer{R_\frakp}}(K,\Delta_e K)$
		as an $(A^\op,A)$-bimodule
		by  setting $a^\op \cdot x\cdot a' = a^\sigma x a' $.
	\end{enumerate}
\end{lem}

\begin{proof}	
	(i) 
	The map $\Phi_{P,Q}$
	is a composition
	of the base-chance map $A\otimes \HH_0(P,\Delta_e Q)\to \Hom(A\otimes P,A\otimes \Delta_e Q)$
	and
	$\Hom(P_A,t_Q)$. The first map is an isomorphism
	by 	Lemma~\ref{LM:derived-cat-base-change} and the second is an isomorphism
	because $t_Q$ is.
	(Recall that
	$\bDer{R_\frakp}$ 
	and $\bDer{\rmod{R_\frakp}}$ are equivalent because $R_\frakp$ is regular,
	and $\bDer{A_\frakp}$ is equivalent to $\bDer{\rmod{A_\frakp}}$
	by Corollary~\ref{CR:bDer-iso}.)
	

	(ii) 
	The map $\Phi_{K,K}$ is an isomorphism by (i), because
	$K,\Delta_e K\in \calC(R_\frakp)$.
	It  is a  morphism of
	right $A$-modules
	because 
	$t_Q\circ ( \ell_{ab}\otimes u)=(t_Q\circ (  \ell_a \otimes u))\circ (  \ell_b\otimes \id)$
	for all $a,b\in A$.
	If we check that $\sigma\otimes \id$ corresponds to $\theta(K)$
	under $\Phi_{K,K}$, then this will imply that $\Phi_{K,K}$ is also 
	a  morphism of left $A^\op$-modules,
	because $a^\op \cdot z = (z^\theta a)^\theta$ for all $z\in Z$, $a\in A$,
	and likewise for $A\otimes \Hom_{\bDer{R_\frakp}}(K,\Delta_e K)$
	and $\sigma\otimes \id$.

	To see that $\sigma\otimes \id$ corresponds to $\theta(K)$
	under $\Phi_{K,K}$, define $\iota:\Hom_{\bDer{R_\frakp}}(K,\Delta_eK)\to
	\Hom_{\bDer{R_\frakp}}(K,\Delta_eK)$ by $\iota(u)=\Delta_e u\circ \zeta_{e,K}$.
	It is routine to check that 
	$\vphi\mapsto D_e\vphi\circ \omega_{e,K_A}$
	corresponds to $ \sigma  \otimes \iota$ under $\Phi_{K,K}$,
	so it is enough
	to show that $\iota$ is the identity.
	It is easy to check that $\iota$ is $R_\frakp$-linear
	and satisfies $\iota\circ \iota=\id$.
	Since  $\Hom_{\bDer{R_\frakp}}(K,\Delta_eK)\cong
	\Hom(\HH_0(K),\HH_0(\Delta_eK))=\Hom(k(\frakp),\Ext^e_{R_\frakp}(k(\frakp),R_\frakp))\cong k(\frakp)$,
	this means that $\iota\in \{\pm \id\}$.
	To finish, it remains
	to exhibit a nonzero
	$\vphi:K\to \Delta_eK$
	such that $\vphi= \iota(\vphi)$.	
	
	It is straightforward to check that $\iota$ is a natural transformation
	from the contravariant functor $K\mapsto \Hom_{\bDer{R_\frakp}}(K,\Delta_eK)$
	to itself, so we may replace $K$ with any 
	chain complex   quasi-isomorphic to it.
	In particular, we may assume that $K$ is
	the    Koszul complex $K(E,s)$  associated to
	a  regular
	sequence in $R_\frakp$ generating $\frakp_\frakp$; see~\cite[Proposition~1.9]{Bayer_2019_Gersten_Witt_complex_prerprint}
	and the preceding comment,
	for instance.
	Fix an isomorphism $\alpha:\bigwedge^e E\to R_\frakp$,
	and define $\vphi:K\to \RD_e K $ by $(\vphi_i k)k'=(-1)^{ei-\frac{1}{2}i(i-1)}\alpha({k\wedge k'})$
	for all $k\in  \bigwedge^i E$, $k'\in  \bigwedge^{e-i} E$.
	One readily checks that $\vphi\in\Hom_{\bDer{R_\frakp}}(K,\Delta_eK)$,
	$\vphi\neq 0$ and $\iota(\vphi)=\vphi$.
\end{proof}

Now, 
by the discussion in \ref{subsec:homological},
we have an isomorphism
$
\tilde{k}(\frakp) = \Ext_{R_\frakp}^e(k(\frakp),R_\frakp)\cong
\HH_0(\RD_e K )
$. This gives rise to an isomorphism
\begin{align}\label{EQ:k-frakp-iso}
\Hom_{\bDer{R}}(K,\RD_e K)\xrightarrow{\,\HH_0}\Hom(k(\frakp),\HH_0(\Delta_e K))
\xrightarrow{\phi\mapsto \phi(1) }\HH_0(\Delta_eK)\cong \tilde{k}(\frakp),
\end{align}
denoted $\beta^{(K)}$.
By Lemma~\ref{LM:Phi-P-Q-dfn}, 
this   induces an $(A^\op,A)$-module isomorphism
\begin{equation}
\label{EQ:Phi-P}
\Phi^{(K)}:=\Phi_{K,K}\circ (\id \otimes\beta^{(K)})^{-1}: 
A \otimes \tilde{k}(\frakp) 
\to
\Hom_{\bDer{A_\frakp}}(K_A,D_eK_A)=Z(K)
\end{equation}
under which
$\sigma \otimes\id_{\tilde{k}(\frakp)}$
corresponds to  
$\theta(K)$.
The map $(\Phi^{(K)})^{-1}$ is the desired
identification of $(Z(K),\theta(K))$
with $(A \otimes \tilde{k}(\frakp),\sigma\otimes\id_{\tilde{k}(\frakp)})$.
Combining this with Proposition~\ref{PR:Koszul-transfer},
we get an isomorphism
\begin{align}\label{EQ:dfn-of-Fp}
\tilde{F}^{(K)}:
W_\veps(\calC(A_\frakp),D_e,\omega_e)
\to W_\veps(A(\frakp),\sigma(\frakp);\tilde{k}(\frakp))= \tilde{W}_\veps(A(\frakp)).
\end{align}

\begin{prp}
	With the above notation, $\tilde{F}^{(K)}$ is independent of
	the projective resolution $K$ of $k(\frakp)$.
\end{prp}

\begin{proof}
	Let $K'$ be another projective resolution of $k(\frakp)$ supported in degrees $0,\dots,e$,
	and let  $(P,f)\in\Herm[\veps]{\calC(A_\frakp)}$.
	We need to prove that $\tilde{F}^{(K)}[P,f]=\tilde{F}^{(K')}[P,f]$.
	
	By unfolding the definitions, we see that
	that $\tilde{F}^{(K)}[P,f]$ is represented by 
	the hermitian space $(\Hom_{\bDer{A_\frakp}}(K_A,P),g)$,
	where $\hat{g}$
	is the $A\otimes\tilde{k}(\frakp)$-valued
	pairing on $\Hom_{\bDer{A_\frakp}}(K_A,P)$ given by
	\[
	\hat{g}(\vphi,\psi)= (\Phi^{(K)})^{-1}(D_e\vphi\circ \veps f\circ \psi).
	\]
	Likewise, $\tilde{F}^{(K')}[P,f]$ is represented by
	$(\Hom_{\bDer{A_\frakp}}(K'_A,P),g')$ with $\hat{g}'$ given by
	$
	\hat{g}'(\vphi,\psi)= (\Phi^{(K')})^{-1}(D_e\vphi\circ \veps f\circ \psi)
	$ for all $\vphi,\psi\in \Hom_{\bDer{A_\frakp}}(K'_A,P)$.
	
	By 	\cite[Theorem~12.4]{Buhler_2010_exact_categories}
	there exists an isomorphism $\xi:K\to K'$ in $\bDer{R_\frakp}$
	such that $\HH_0(\xi)=\id_{k(\frakp)}$. 
	Then $\xi_A:K_A\to K'_A$ is an isomorphism
	in $\bDer{A_\frakp}$. 
	We will show that $\Hom_{\bDer{A_\frakp}}(\xi_A,P)$ is an isometry
	from $g'$ to $g$, thus proving that $\tilde{F}^{(K)}[P,f]=\tilde{F}^{(K')}[P,f]$.
	By the formulas for $g$ and $g'$, this amounts to showing
	that for all $\vphi,\psi\in\Hom_{\bDer{A_\frakp}}(K'_A,P)$, we have
	\begin{align*}
	(\Phi^{(K')})^{-1}(D_e\vphi\circ \veps f\circ \psi)
	&=
	(\Phi^{(K)})^{-1}(D_e(\vphi \circ \xi_A)\circ \veps f\circ (\psi\circ \xi_A)) .
	\end{align*}
	Since the right hand side equals 
	$(\Phi^{(K)})^{-1}(D_e\xi_A \circ (D_e\vphi\circ \veps f\circ \psi)\circ \xi_A)$,
	it is enough to check that
	$(\Phi^{(K')})^{-1}=
	(\Phi^{(K)})^{-1}\circ \Hom_{\bDer{A_\frakp}}(\xi_A,D_e\xi_A)$,
	or equivalently, that
	\begin{align}\label{EQ:equation-to-check-in-Prp}
	\Hom_{\bDer{A_\frakp}}(\xi_A,D_e\xi_A)\circ \Phi^{(K')}=
	\Phi^{(K)}.
	\end{align}

	By \eqref{EQ:Phi-P}, 
	in order to show \eqref{EQ:equation-to-check-in-Prp},	
	it is enough to show
	that the diagram
	\[
	\xymatrixcolsep{3.5pc}\xymatrix{
	A\otimes \tilde{k}(\frakp) \ar[r]^-{\id\otimes(\beta^{(K)})^{-1}} 
		\ar[rd]_-{\id\otimes (\beta^{(K')})^{-1}} &
	A \otimes\Hom_{\bDer{R_\frakp}}(K,\Delta_e K) 
		\ar[r]^{\Phi_{K,K}} &
	\Hom_{\bDer{A_\frakp}}(K_A,D_eK_A) \\
	&
	A \otimes \Hom_{\bDer{R_\frakp}}(K',\Delta_e K') 
		\ar[u]|{\id\otimes\Hom_{\bDer{R_\frakp}}(\xi,\Delta_e\xi)} 
		\ar[r]^{\Phi_{K',K'}} & 
	\Hom_{\bDer{A_\frakp}}(K'_A,D_eK'_A) 
		\ar[u]|{\Hom_{\bDer{A_\frakp}}(\xi_A,D_e\xi_A)}
	}	
	\]
	commutes. The right square commutes by the naturality
	of $\Phi_{-,-}$ in both inputs, so it is enough to prove that the triangle on the left
	commutes. This will follow if we show that 
	$\beta^{(K)}\circ \Hom_{\bDer{R_\frakp}}(\xi,\Delta_e \xi)
	=\beta^{(K')}$. By the definition of $\beta^{(K)}$
	in \eqref{EQ:k-frakp-iso},
	this will follow once we show that the diagram
	\[
	\xymatrix{
	\Hom_{\bDer{R_\frakp}}(K,\Delta_e K) \ar[r]^{\HH_0} &
	\Hom(k(\frakp),\HH_0(\Delta_e K)) \ar[r]^-{\phi\mapsto \phi(1)} &	
	\HH_0(\Delta_e K) \ar[r]^-{u^{-1}} &
	\tilde{k}(\frakp) \\
	\Hom_{\bDer{R_\frakp}}(K',\Delta_eK') \ar[r]^{\HH_0} 
		\ar[u]|{\Hom_{\bDer{R_\frakp}}(\xi,\Delta_e\xi)} &
	\Hom(k(\frakp),\HH_0(\Delta_e K')) \ar[r]^-{\phi\mapsto \phi(1)}
		\ar[u]|{\Hom(\id_{k(\frakp)},\HH_0(\Delta_e\xi))} &
	\HH_0(\Delta_e K') \ar[ru]_-{u'^{-1}} \ar[u]|{\HH_0\Delta_e\xi}
	}
	\]
	commutes.
	Here, $u$ is the isomorphism $\Ext^e_{R_\frakp}(k(\frakp),R_\frakp)\to \HH_{-e}(\Hom( K,R)) $,
	defined as in \ref{subsec:homological}, and similarly for $u'$.
	The left square commutes because the functor $\HH_0$ respects composition
	and $\HH_0(\xi)=\id_{k(\frakp)}$. That the middle square commutes
	is a straightforward computation. Finally, we noted in \ref{subsec:homological}
	that the  triangle on the right commutes. This completes the proof.
\end{proof}

We conclude by defining $s_e$ for $e\geq 0$.

\begin{construction}\label{CN:GW-isomorphism}
	For all $e\in\N\cup\{0\}$, define
	\[
	s'_e:\iGW{A,\sigma,\veps}_e= W_\veps^e(\bDer[e]{A}/\bDer[e+1]{A})
	\xrightarrow{\sim}\bigoplus_{\frakp\in R^{(e)}} \tilde{W}_\veps(A(\frakp) )=\GW{A,\sigma,\veps}_e 
	\]
	as the composition of the isomorphisms
	\begin{align*}
	\loc : & W_\veps^e(\bDer[e]{A}/ \bDer[e+1]{A}) 
	\to 
	\bigoplus_{\frakp\in R^{(e)}}W_\veps^e(\bDer[e]{A_\frakp}),
	\\
	\text{nat.map}^{-1}:&\bigoplus_{\frakp\in R^{(e)}}W_\veps^e(\bDer[e]{A_\frakp})
	\to
	\bigoplus_{\frakp\in R^{(e)}}W_\veps(\calC^0(A_\frakp)),
	\\
	\bigoplus F^{(\frakp)}: &
	\bigoplus_{\frakp\in R^{(e)}}W_\veps(\calC^0(A_\frakp))
	\to
	\bigoplus_{\frakp\in R^{(e)}}\tilde{W}_\veps(A(\frakp) ),
	\end{align*}
	constructed in    Propositions~\ref{PR:derived-localization} and \ref{PR:GW-reduction-I},
	and \eqref{EQ:dfn-of-Fp}, where
	we have written $F^{(\frakp)}$ for $\tilde{F}^{(K)}$ as it is independent
	of $K$. The second map is well-defined
	by Proposition~\ref{PR:GW-reduction-II}.
	Finally, set
	$s_e=(-1)^es'_e$.
\end{construction}

\section{Proof of Theorem~\ref{TH:GW-well-defined}}
\label{sec:proof-of-dd-desc}

We  use the notation of Sections~\ref{sec:GW-Balmer}
and \ref{sec:GW-propositions}.

Denote the $e$-th differentials of $\iaGW{A,\sigma,\veps}$
and $\aGW{A,\sigma,\veps}$ by $\dd_e$ and $\dd'_e$,
respectively. We need to show
that $\dd'_e\circ s_e=s_{e+1}\circ \dd_e$. The case $e=-1$
is routine and is left to the reader. We  
assume that $e\geq 0$ henceforth.
In this case, the theorem is equivalent
to showing that for every $\frakp\in R^{(e)}$, $\frakq\in R^{(e+1)}$
and $(V,f)\in\Herm[\veps]{A(\frakp),\sigma(\frakp);\tilde{k}(\frakp)}$
(notation as in Example~\ref{EX:line-bundle}),
we have
$
(s_{e+1}\dd_e s^{-1}_e[V,f])_{\frakq}=\partial_{\frakp,\frakq}[V,f]$,
or equivalently,
\begin{align}\label{EQ:what-to-prove}
(s'_{e+1}\dd_e s'^{-1}_e[V,f])_{\frakq}=-\partial_{\frakp,\frakq}[V,f],
\end{align}
with the convention that $\partial_{\frakp,\frakq}=0$ if
$\frakp\nsubseteq \frakq$.

\begin{proof}[Proof of \eqref{EQ:what-to-prove} when $\frakp\nsubseteq \frakq$.]
	By Proposition~\ref{PR:derived-localization},
	there is an isomorphism
	$\iGW{A,\sigma,\veps}_e\cong \bigoplus_{\fraka\in R^{(e)}}W^e_\veps(\bDer[e]{A_\fraka})$
	for all $e\geq 0$.
	Under this isomorphisms, the map $\dd_e:\iGW{A,\sigma,\veps}_e \to \iGW{A,\sigma,\veps}_{e+1}$
	corresponds to a map
	$\tilde{\dd}_e: \bigoplus_{\fraka\in R^{(e)}}W^e_\veps(\bDer[e]{A_\fraka})\to 
	\bigoplus_{\frakb\in R^{(e+1)}}W^{e+1}_\veps(\bDer[e]{A_\frakb})$.
	By Construction~\ref{CN:GW-isomorphism}, it is enough to show that the composition
	\begin{equation*} 
	W^e_\veps(\bDer[e]{A_\frakp})\xrightarrow{i}
	\bigoplus_{\fraka\in R^{(e)}}W^e_\veps(\bDer[e]{A_\fraka})\xrightarrow{\tilde{\dd}_e} 
	\bigoplus_{\frakb\in R^{(e+1)}}W^{e+1}_\veps(\bDer[e+1]{A_\frakb})
	\xrightarrow{p}
	W^{e+1}_\veps(\bDer[e+1]{A_\frakq}),
	\end{equation*}
	in which $i$ and $p$ are the evident embedding and projection,
	is $0$. 
	
	We show this by comparing to the Gersten--Witt complex of $(A_\frakq,\sigma_\frakq,\veps)$.
	Indeed, localization-at-$\frakq$ defines a functor $\bDer{A}\to\bDer{A_\frakq}$
	which respects the  
	triangulated hermitian structures
	and codimension filtrations \eqref{EQ:codim-filtration} on $\bDer{A}$
	and $\bDer{A_\frakq}$. Thus, it induces a morphism 
	$\iGW{A,\sigma,\veps}\to\iGW{A_\frakq,\sigma_\frakq,\veps}$.
	Applying  Proposition~\ref{PR:derived-localization} to $A_\frakq$,
	we see
	that $\iGW{A_\frakq,\sigma_\frakq,\veps}_e\cong \bigoplus_{\fraka\in R_\frakq^{(e)}}
	W^e_\veps(\bDer[e]{(A_\frakq)_\fraka})=\bigoplus_{\fraka\in R_\frakq^{(e)}:\fraka\subseteq \frakq}
	W^e_\veps(\bDer[e]{A_\fraka})$, and likewise $\calB^{A_\frakq,\sigma_\frakq,\veps}_{e+1}
	\cong W_\veps^{e+1}(\bDer[e+1]{A_\frakq})$.
	This gives rise to a commutative diagram 
	\[
	\xymatrix{
	\bigoplus_{\fraka\in R^{(e)}}W^e_\veps(\bDer[e]{A_\fraka}) \ar[r]^-{\tilde{\dd}_e} \ar[d]^{p'} & 
	\bigoplus_{\frakb\in R^{(e+1)}}W^{e+1}_\veps(\bDer[e+1]{A_\frakb}) \ar[d]^{p} 
	\\
	\bigoplus_{\fraka\in R^{(e)}:\fraka\subseteq\frakq}W^e_\veps(\bDer[e]{A_\fraka}) \ar[r]  &
	W_\veps^{e+1}(\bDer[e+1]{A_\frakq}) 
	}
	\]
	in which 
	$p'$ and $p$  are induced by the morphism 
	$\iGW{A,\sigma,\veps}\to\iGW{A_\frakq,\sigma_\frakq,\veps}$.
	It is straightforward to check that $p'$ and $p$ are just the evident
	projections (so $p$ agrees with $p$ from the previous paragraph). 
	Since $\frakp\nsubseteq \frakq$, this implies that
	$p'\circ p=0$, and thus $i\circ\tilde{\dd}_e\circ p=0$.
\end{proof}

We henceforth set our attention to the case
$\frakp\subseteq \frakq$.
It is easy to see that localizing at $\frakq$ is harmless,
so we further restrict to the case where $R$ is regular local and $\frakq$ is its maximal
ideal. 
We now evoke the notation   used in the construction
of $\partial_{\frakp,\frakq}$ in Section~\ref{sec:GW-second-res}. In
particular, 
\begin{align*}
S& =R/\frakp, & \tilde{S}&=\Ext^e_R(S,R), \\
\frakm &= \frakq/\frakp, &  
\tilde{\frakm}^{-1}&=\Ext^e_R(\frakm,R).
\end{align*}  
We also let $A_S=A\otimes S$ and $\sigma_S=\sigma\otimes\id_S$.
Note that $S$ is a $1$-dimensional domain,   and thus Cohen-Macaulay as an 
$R$-module. By \cite[Proposition~1.3(i)]{Bayer_2019_Gersten_Witt_complex_prerprint}, 
the modules $\tilde{S}$ and $\tilde{\frakm}^{-1}$
are $S$-torsion-free. Recall that localization-at-$\frakp$
defines an embedding of both $\tilde{S}$ and $\tilde{\frakm}^{-1}$
in $\tilde{k}(\frakp)=\Ext^e_{R_\frakp}(k(\frakp),R_\frakp)$,
and $\tilde{S}\subseteq \tilde{\frakm}^{-1}$ as $R$-submodules
of $\tilde{k}(\frakp)$.

\medskip

We introduce additional notation:
\begin{enumerate}[label=(n\arabic*)]
	\item $\calD=\bDer{A}$. 
	As in Section~\ref{sec:GW-Balmer}, the shifted hermitian structures
	of $\calD $ induced by $\sigma $ are   denoted
	$(D_n,\delta_n,\omega_n)_{n\in \Z}$. 
\end{enumerate}	
Contrary to the convention elsewhere in this work, unadorned $\Hom$-sets
are  taken in $\calD$.
If $\alpha:M\to N$ is a morphism in $\calD$,
we write $\Hom(\alpha,-):\Hom(N,-)\to \Hom(M,-)$ as $\alpha^\diamond$
and $\Hom( -,\alpha):\Hom(-,M)\to \Hom(-,N)$ as $\alpha_\diamond$;
these natural transformations are given by pre-composition
and post-composition with $\alpha$, respectively.
\begin{enumerate}[resume, label=(n\arabic*)]
	\item $\calD'=\bDer{A_\frakp}$. The shifted hermitian structures
	of $\calD'$ induced by $\sigma_\frakp$ are   denoted
	$(D'_n,\delta_n,\omega'_n)_{n\in \Z}$.
	\item  The shifted hermitian structures
	of $\bDer{R}$ induced by the involution $\id_R$ are   denoted
	$(\Delta_n,\delta_n,\zeta_n)_{n\in \Z}$ (as in Section~\ref{sec:GW-propositions}).
	\item $\calA$ is the full subcategory
of $\calD$ consisting of 
chain complexes $P$
supported in non-negative degrees and
satisfying $\HH_0(P)\in\rmod{A_S}$ and $\HH_i(P)=0$ for all $i\neq 0$;
	this is an abelian category, because
	$\HH_0:\calA\to \rmod{A_S}$ is an equivalence
	(Lemma~\ref{LM:global-dim-of-Az}, \cite[Theorem~12.4]{Buhler_2010_exact_categories}).
	
	\item $\calL$ is the full subcategory
	of $\rmod{A_S}$ consisting of $A_S$-modules which are $S$-torsion-free.
	It inherits an exact structure from $\rmod{A_S}$.
	
	\item $\calB$ is the full subcategory
	of $\calA$ consisting of complexes $P\in \calA$
	supported in degrees
	$0,\dots,e$ and
	satisfying $H_0(P)\in\calL$. It inherits an exact structure from the abelian category 
	$\calA$.
	
	\item \label{item:tilde-star} $(\tilde{*},\tilde{\omega})$
	is the exact hermitian structure on $\calL$
	obtained by  applying   Example~\ref{EX:line-bundle} to $(A_S,\sigma_S)$
using the \emph{possibly non-invertible} $S$-module $\tilde{S}$
and allowing modules in $\calL$, rather than $\rproj{A_S}$. 
	This is an  exact hermitian structure by 
	\cite[Proposition~1.5, Corollary~1.8]{Bayer_2019_Gersten_Witt_complex_prerprint}
	($\tilde{S}$ is a dualizing module for $S$ by \cite[Theorem~3.3.7(b)]{Bruns_1993_cohen_macaulay_rings}).	
	
	\item \label{item:L} $L$ is a minimal projective resolution of the $R$-module $S$. 
	\item \label{item:K} $K$ is a minimal projective resolution of the $R$-module $k(\frakq)$.
	\item \label{item:J} $J$ is a minimal projective resolution of the $R$-module $\frakm:=\frakq/\frakp$.
	\item $L_A=A\otimes L$, $K_A=A\otimes K$ and $J_A=A\otimes J$; they are objects of $\calA$
	because 
	Since $A$ is flat over $R$. We have right $A$-module 
	isomorphisms $\HH_0(L_A)\cong A_S$ and $\HH_0(K_A)\cong A(\frakq)$,
	which we suppress.
	\item \label{item:J-L-K-A-triangle} 
	The short exact sequence $\frakm\embeds S\onto k(\frakq)$ gives rise to a distinguished
	triangle
	\[
	J  \xrightarrow{\iota_0} L \xrightarrow{p_0} K \xrightarrow{q_0} TJ  
	\] in $\bDer{R}$.
	Tensoring with $A$ gives rise to a distinguished
	triangle in $\calD$:
	\begin{equation*}
	J_A \xrightarrow{\iota} L_A\xrightarrow{p} K_A\xrightarrow{q} TJ_A.
	\end{equation*}
	
	\item \label{item:F-functor} $F:=\Hom(L_A,-):\calA\to \rmod{A_S}$.
	The right $A_S$-module structure on $\Hom(L_A,P)$  arises
	from the $R$-algebra isomorphism  
	\[\End(L_A)\xrightarrow[\HH_0]{\sim} 
	 \End_A(A_S)\xrightarrow[\phi\mapsto\phi(1)]{\sim} A_S.\]
	Explicitly, if $\quo{a}\in A_S$ lifts to $a\in A$ and $f\in \Hom(L_A,P)$,
	then $ f\cdot  \quo{a} =f\circ (\ell_a\otimes\id_L)$, where $\ell_a:A\to A$
	is left-multiplication by $a$.  
	
	\item \label{item:Hom-L-DL} 
	We endow $\Hom(L_A,D_eL_A)$ with an $(A_S^\op,A_S)$-bimodule structure
	as follows: if $\quo{a},\quo{b}\in A_S$ lift to $a,b\in A$,
	then $\quo{a}^\op \cdot f\cdot  \quo{b} =D_e(\ell_a\otimes\id_L)\circ f\circ (\ell_b\otimes\id_L)$
	with $\ell_a$   as in \ref{item:F-functor}.  
	
	\item \label{item:Phi-L} As in \eqref{EQ:k-frakp-iso},
	there is an isomorphism $\beta:\Hom_{\bDer{R}}(L,\Delta_eL)\to \Ext^e_R(S,R)= \tilde{S}$.
	We use   $\Phi_{-,-}$ defined in Lemma~\ref{LM:Phi-P-Q-dfn}(i) 
	and $\beta$ to construct an 
	$R$-module isomorphism
	\begin{equation*}
\Phi^{(L)}:=\Phi_{L,L}\circ (\id\otimes \beta)^{-1}:A\otimes \tilde{S} \to \Hom(L_A,D_eL_A).
	\end{equation*}
	
	\item \label{item:K-prime}
	$K'=L_\frakp\in \bDer{R_\frakp}$. 
	It  is a projective resolution of the $R_\frakp$-module $k(\frakp)$.
	We shall see below (Lemma~\ref{LM:length-of-L}) 
	that $L$, and hence $K'$, is supported in degrees $0,\dots,e$,
	so $K'\in\calC^0(A_\frakq)$.
	Comparing \ref{item:Phi-L} with
\eqref{EQ:Phi-P} shows that
	\[ (\Phi^{(L)})_{\frakp}=\Phi^{(K')}:A\otimes \tilde{k}(\frakp)\to \Hom_{\bDer{A_\frakp}}(K'_A,D_eK'_A),  \]
	where $\Phi^{(K')}$ is as in  Section~\ref{sec:GW-propositions}.\footnote{
		Here, we are   using the fact
		that forming $\Hom$-sets and $\Ext$-groups of finitely presented $A$-modules
		commutes with localization at $\frakp$ up to a natural isomorphism,
		see  \cite[Theorems~2.38, 2.39]{Reiner_2003_maximal_orders_reprint}.
		Also, taking homology   commutes with localization at $\frakp$,
		because $R_\frakp$ is flat over $R$.
}
	In particular, we have an identification of  $(A(\frakp)^\op,A(\frakp))$-bimodules
	\[\Hom(L_A,D_eL_A)_\frakp= \Hom_{\calD'}(K'_A,D'_eK'_A).\]
	
\end{enumerate}

\medskip

We shall need several lemmas.

\begin{lem}\label{LM:F-H0-equiv-lem}
	There is a natural isomorphism 
	\[FP=\Hom(L_A,P)\cong \HH_0(P)\] for every 
	$P\in \calA$. 
	Moreover, there is an natural isomorphism $\Hom(K_A,Q)\cong \HH_0(Q)$
	for every $Q\in \calA$ with $\HH_0(Q)\cdot \frakq=0$
	such that  the induced
	isomorphism $\Hom(K_A,Q)\to \HH_0(Q)\to\Hom(L_A,Q)$ is $p^\diamond$
	(with $p$ as in~\ref{item:J-L-K-A-triangle}).
\end{lem}

\begin{proof}
	Recall from  \ref{subsec:homological} that the data of $L$
	includes a map $ \alpha_L:L_0\to S$ which we use to identify 
	$\HH_0(L)$ with $R$. Since $A$ is flat over $R$, this gives rise
	to an isomorphism $\HH_0(L_A)\to A_S$, denoted $\quo{(\alpha_L)_A}$.
	Let  $\psi_P: FP=\Hom(L_A,P)\to \HH_0(P)$ denote the composition
	\begin{equation}\label{EQ:Phi-L}
	\Hom (L_A,P)\xrightarrow{\HH_0}\Hom_{A }(A_S,\HH_0(P))\xrightarrow{\phi\mapsto \phi(1)} \HH_0(P),
	\end{equation}
	where in the first arrow, we used $\quo{(\alpha_L)_A}$ to
	identify
	$\HH_0(L_A)$ with  $A_S$.
	Explicitly, given $f:L_A\to P$, we have $\psi_P(f)=(\HH_0(f)\circ (\quo{(\alpha_L)_A})^{-1})(1_{A_S})$.
	The left arrow in \eqref{EQ:Phi-L} is a bijection because
	$\HH_0:\calA\to \rmod{A_S}$ is an equivalence, and the right
	arrow is a bijection because $\HH_0(P)\cdot  \frakp=0$
	and $A_S=A/A\frakp$.	
	Thus, $\psi_P$ is an isomorphism,
	and it is routine
	to check that it is natural in $P$.

	The natural homomorphism $\Hom(K_A,Q)\to \HH_0(Q)$ --- call it $\psi'_Q$ --- is defined similarly, 
	using the isomorphism 
	$\quo{(\alpha_K)_A} :\HH_0(K_A)\to A(\frakq)$ induced by $\alpha_K:K_0\to k(\frakq)$.
	It is an isomorphism because $\HH_0(Q)\cdot \frakq=0$.
	
	Checking the last assertion amounts to showing that $\psi'_Q=\psi_Q\circ p^\diamond$.
	Observe that 
	in the construction of the first distinguished triangle  in \ref{item:J-L-K-A-triangle},
	$p_0:L\to K$ is taken to be compatible with $\alpha_L:L_0\to S$
	and $\alpha_K:K_0\to k(\frakq)$, and thus $ \quo{(\alpha_K)_A} \circ \HH_0(p)= \gamma\circ \quo{(\alpha_L)_A} $, where $\gamma$ is the quotient map $A_S\to A(\frakq)$.
	Now,  for all $f:K_A\to P$, we have $\psi'_Q(f)=(\HH_0(f)\circ (\quo{(\alpha_K)_A})^{-1})(1_{A(\frakq)})
	=(\HH_0(f)\circ \HH_0(p)\circ (\quo{(\alpha_L)_A})^{-1})(1_{A_S})=
	(\HH_0(f\circ p)\circ (\quo{(\alpha_L)_A})^{-1})(1_{A_S})=\psi_Q(p^\diamond(f))$,
	as required.
\end{proof}

\begin{lem}\label{LM:duality-on-B}
	$(D_e,\omega_e)$ restricts to an exact hermitian structure on $\calB$.
\end{lem}

\begin{proof}
	By \cite[Propositions~1.5, 1.7]{Bayer_2019_Gersten_Witt_complex_prerprint},
	for all $M\in\calL$, we have $\Ext^i_A(M,A)=0$ if $i\neq e$,
	and $\Ext^e_A(M,A)$ is $S$-torsion-free, so $(D_e,\omega_e)$
	restricts to a hermitian structure on $\calB$. 
	To see that $D_e$ is exact,
	consider a short exact sequence $P'\embeds P\onto P''$   in $\calB$.
	Since $\Ext^{e-1}_A(\HH_0(P'),A)=0$
	and $\Ext^{e+1}_A(\HH_0(P''),A)=0$,
	we have a short exact sequence
	$\Ext^e_A(\HH_0(P''),A)\embeds
	\Ext^e_A(\HH_0(P),A)\onto \Ext^e_A(\HH_0(P'),A)$.
	This sequence is isomorphic to 
	$\HH_0(D_eP'')\to  \HH_0(D_eP)\to \HH_0(D_eP')$, so the latter
	is also short exact sequence in $\rmod{A_S}$.
	Since  $\HH_0:\calA\to \rmod{A_S}$ is an exact equivalence,
	it follows that $D_eP'\to D_eP\to D_eP''$ is a
	short exact sequence in $\calB$.
\end{proof}

\begin{lem}\label{LM:B-L-equiv}
	Both $\HH_0$ and $F=\Hom(L_A,-)$ restrict
	to an exact equivalence from $\calB$ to $\calL$.
\end{lem}

\begin{proof}
	Since $\HH_0\cong F$ as functors from $\calA$ to $\rMod{A_S}$
	(Lemma~\ref{LM:F-H0-equiv-lem}), it is enough to consider $\HH_0$.
	We already know  that $\HH_0:\calA\to \rmod{A_S}$ is an exact equivalence,
	and $\HH_0$ maps to $\calB$ to $\calL$ by the definition of $\calB$,
	so it remains to check that $\HH_0:\calB\to \calL$ is essentially
	surjective. 
	By \cite[Proposition~1.5]{Bayer_2019_Gersten_Witt_complex_prerprint},
every nonzero finite $S$-torsion-free $S$-module   is Cohen-Macaulay of dimension $1$.
Thus, by  the Auslander--Buchsbaum formula \cite[Theorem~1.3.3]{Bruns_1993_cohen_macaulay_rings} 
and Lemma~\ref{LM:global-dim-of-Az},
every right $A$-module in $\calL$ has projective dimension at most $e$,
and is therefore isomorphic to $\HH_0(P)$ for some $P\in \calB$.
\end{proof}

\begin{lem}\label{LM:length-of-L}
	$L$ and $J$ (see~\ref{item:L} and~\ref{item:J}) are supported in degrees $0,\dots,e$,
	and $L_A,J_A\in \calL$.
\end{lem}

\begin{proof}
	Applying Lemma~\ref{LM:B-L-equiv} with $A=R$ shows that $S$ and $\frakm$
	are in the image of $\HH_0:\calB\to \calL$ (when $A=R$).
	Since   $L$ and $J$ are minimal $R$-projective resolutions of $S$ and $\frakm$,
	respectively, the lemma follows.
\end{proof}

\begin{lem}\label{LM:Phi-L-AS-linear}
	The isomorphism $\Phi^{(L)}:A\otimes\tilde{S}\to \Hom(L_A,D_eL_A)$
	of \ref{item:Phi-L}
	is an $(A_S^\op,A_S)$-bimodule
	isomorphism. See \ref{item:Hom-L-DL} for the module structure on 
	the range;	
	the module structure on $A\otimes\tilde{S}$
	is determined by $a^\op \cdot x\cdot b=a^\sigma xb$ ($a,b\in A$, $x\in A\otimes\tilde{S}$). 
	Under this isomorphism, 
	$\sigma\otimes \id_{\tilde{S}}$
	corresponds to $\vphi\mapsto D_e\vphi \circ \omega_{e,L_A}$.
\end{lem}

\begin{proof}
	Since $\tilde{S}$ is $S$-torsion-free  and $A$ is finite projective
	over $R$, the $S$-module $A\otimes\tilde{S}$ is $S$-torsion-free.
	As $\Phi^{(L)}$ is an $R$-module isomorphism,
	$\Hom(L_A,D_eL_A)$ is also $S$-torsion-free. It is therefore enough
	to prove the lemma after localizing at $\frakp$.
	The
	required statements now follow from Lemma~\ref{LM:Phi-P-Q-dfn}(ii)
	(applied to $K'$)
	and the second equation of \ref{item:K-prime}.
\end{proof}

\begin{lem}\label{LM:equiv-of-L-and-B}
	For $P\in\calB$, define a natural transformation
	$j_P: FD_eP\to (FP)^{\tilde{*}}$
	by
	\[(j_P \vphi)\psi=(\Phi^{(L)})^{-1}(D_e\vphi \circ \omega_{e,P}\circ \psi)\]
	for all $\vphi\in \Hom(L_A,D_eP)$ and $\psi\in \Hom(L_A,P)$.
	Then  $(F,j):(\calB,D_e,\omega_e)\to (\calL,\tilde{*},\tilde{\omega})$
		is an exact $1$-hermitian  equivalence.
\end{lem}

Once suppressing $\Phi^{(L)}$,
the hermitian functor $(F,j)$
is defined using the same formula as the  
$L_A$-transfer functor of Section~\ref{sec:transfer},
so we shall (abusively) refer to it as the $L_A$-transfer functor.
Note, however, that it is not true  in general that every
object of $\calB$ is a summand of a direct sum of copies of $L_A$.

\begin{proof}
	By Lemma~\ref{LM:B-L-equiv}, $F:\calB\to \calL$ is an exact
	equivalence.
	That $j$ is natural and satisfies 
	$j_{D_eP}\circ F\omega_{e,P}= j_P^{\tilde{*}}\circ \tilde{\omega}_{FP}  $
	for all $P\in \calB$ is routine, thanks to
	Lemma~\ref{LM:Phi-L-AS-linear}. 
	It remains to check that $j_P$ is an isomorphism
	for all $P\in \calB$.

	We can find $n,m\in \N$
	and an exact sequence $A_S^n\to A_S^m\to \HH_0(P)\to 0$ in $\calL$.
	Since $\HH_0:\calB\to \calL$ is an exact equivalence
	(Lemma~\ref{LM:B-L-equiv}),
	this gives rise to an exact
	sequence $L_A^n\to L_A^m\to P\to 0$ in $\calB$,
	which in turn gives
	rise to exact sequences
	$0\to FD_e P\to FD_eL_A^m\to FD_eL_A^n$
	and $0\to (FP)^{\tilde{*}}\to (FL_A^m)^{\tilde{*}}\to
	(FL_A^n)^{\tilde{*}}$.
	The natural transformation
	$j$ determines a morphism between
	these   sequences. It is routine to check
	that $j_{L_A}$ is an isomorphism, and hence so
	are $j_{L_A^m}$ and $j_{L_A^n}$.
	By the Five Lemma, $j_P$ is an isomorphism as well.
\end{proof}

\begin{lem}\label{LM:A-tilde-S-diag}
The isomorphisms $\Phi^{(L)}$ of \ref{item:Phi-L},
the isomorphism
$\Phi^{(K )}$ of \eqref{EQ:Phi-P}
(with $K$ as in \ref{item:K})
and the exact sequence \eqref{EQ:fundamental-exact-seq}
fit into a commutative diagram  of
$(A^\op,A)$-bimodules
\begin{equation} 
\xymatrix{
A\otimes \tilde{S} \ar@{^{(}->}[r]   \ar[d]^{\Phi^{(L)}}  &
A\otimes \tilde{\frakm}^{-1}  \ar[r]^{\calT } \ar[d]^{\Phi^{(J)}}  & 
{A}\otimes \tilde{k}(\frakq) \ar[d]^{\Phi^{(K)}} \\
\Hom (L_A,D_eL_A)\ar[r]^{(D_e \iota)_\diamond} &
\Hom (L_A,D_e J_A) \ar[r]^-{t}  &
\Hom (K_A,D_{e+1}K_A) 
}
\end{equation}
in which $\Phi^{(J)}$ is an isomorphism,
and $t$ is the composition
\[
 t:\quad \Hom (L_A,D_eJ)\xrightarrow{(D_{e+1}q)_\diamond} \Hom (L_A,D_{e+1}K_A)
  \xrightarrow{(p^\diamond)^{-1}}  \Hom  (K_A, D_{e+1}K_A)  
\]
(the right arrow is invertible by Lemma~\ref{LM:F-H0-equiv-lem}).
\end{lem}

\begin{proof}
	Recall that $\tilde{S}=\Ext^e_R(S,R)$, $\tilde{\frakm}^{-1}=\Ext^e_R(\frakm,R)$
	and $\tilde{k}(\frakq)=\Ext^{e+1}(k(\frakq),R)$.
	As noted in \ref{subsec:homological}, the  exact sequence 
	\[\tilde{S}\to \tilde{\frakm}^{-1}\to \tilde{k}(\frakq)\]
	is   isomorphic to
	\[
	\HH_0(\Delta_eL)\xrightarrow{\HH_0(\Delta_e \iota_0)}
	\HH_0(\Delta_e J) \xrightarrow{\HH_0(\Delta_{e}T^{-1}q_0)}
	\HH_0(\Delta_e T^{-1} K)=\HH_0(\Delta_{e+1}K) 
	\]
	(see \ref{item:J-L-K-A-triangle}).
	Applying Lemma~\ref{LM:F-H0-equiv-lem} with $A=R$,
	we get a get a commutative diagram
	\[
	\xymatrix{
	\HH_0(\Delta_eL) \ar[d] \ar[r]^{\HH_0(\Delta_e \iota_0)} &
	\HH_0(\Delta_e J) \ar[d] \ar[r]^{\HH_0(\Delta_{e}T^{-1}q_0)} &
	\HH_0(\Delta_{e+1}K) \ar[d] \ar[dr] & \\
	\Hom(L,\Delta_eL) \ar[r]^{(\Delta_e \iota)_\diamond } &
	\Hom(L,\Delta_eJ) \ar[r]^{(\Delta_e T^{-1}q_0)_\diamond} &
	\Hom(L,\Delta_eT^{-1} K)  &
	\Hom(K,\Delta_eT^{-1}K) \ar[l]_{p_0^\diamond}
	} 
	\]
	in which the vertical and diagonal arrows are the natural isomorphisms
	provided by the lemma  and   the $\Hom$-groups were taken in $\bDer{R}$.

	By Lemma~\ref{LM:Phi-P-Q-dfn}(i)
	(applied with $\frakq$ in place of $\frakp$ and 
	$(P,Q)\in\{(L,L),(L,J)$, $(L,T^{-1}K)$, $( K,T^{-1}K)\}$),
	the tensor of the bottom row of the last diagram 
	with $A$ is isomorphic via $\Phi_{-,-}$ to 
	\begin{align*}
	\Hom(L_A,D_eL_A)\xrightarrow{(D_e\iota)_\diamond}
	\Hom(L_A,&D_eJ_A)\xrightarrow{(D_eT^{-1}q)_\diamond} \\
	&\Hom(L_A,D_eT^{-1}K_A) \xleftarrow{p^\diamond}
	\Hom(K_A,D_eT^{-1}K_A).
	\end{align*}
	Putting everything together, we arrive at the following commutative diagram,
	in which the vertical maps are isomorphism.
	\[
	\xymatrixcolsep{4pc}\xymatrix{
A\otimes \tilde{S} \ar[r]   \ar[d] &
A\otimes \tilde{\frakm}^{-1}  \ar[r]  \ar[d]   & 
{A}\otimes \tilde{k}(\frakq) \ar[d]  \\
\Hom (L_A,D_eL_A)\ar[r]^{(D_e \iota)_\diamond} &
\Hom (L_A,D_e J_A) \ar[r]^-{(p^{\diamond})^{-1}  (D_eT^{-1}q)_\diamond}  &
\Hom (K_A,D_{e+1}K_A) 
}
	\]
	Comparing the construction of the left-most, resp.\ right-most, vertical arrow 
	with that of $\Phi^{(L)}$, resp.\  $\Phi^{(K)}$,
	shows that they coincide, so we have obtained the desired commutative diagram. 
\end{proof}

\begin{lem}\label{LM:faithfulness}
	The functor $P\mapsto P_\frakp:\calB\to \calC^0(A_\frakp)$ (notation
	as in Section~\ref{sec:GW-propositions})
	is faithful. 
\end{lem}

\begin{proof}
	Note first that   any $P\in \calB$ is supported
	in degrees $0,\dots,e$ and satisfies $\HH_i(P)=0$
	for $i\neq 0$ and $\HH_0(P)\in \rmod{A_S}$.
	Since $R_\frakp$ is flat over $R$, we have $\HH_i(P_\frakp)=\HH_i(P)_\frakp$
	and thus $P_\frakp\in \calC^0(A_\frakp)$.
	
	There is a   diagram of functors
	\[
	\xymatrixcolsep{4pc}\xymatrix{
	\calB \ar[d]^{\HH_0} \ar[r]^{P\mapsto P_\frakp} & 
	\calC^0(A_\frakp) 	\ar[d]^{\HH_0} \\
	\calL \ar[r]^{M\mapsto M_\frakp} & \rmod{A(\frakp)}
	}
	\]
	which commutes up to natural isomorphism. By Lemma~\ref{LM:B-L-equiv}
	and the discussion in Section~\ref{sec:GW-propositions}, both vertical
	arrows are equivalences, and the bottom arrow is faithful because
	$\calL$ consists of $S$-torsion-free $S$-modules. Thus, the top arrow
	is also a faithful functor.	
\end{proof}

We are now ready to establish \eqref{EQ:what-to-prove}
in the case $\frakp\subseteq\frakq$. This will complete the proof of Theorem~\ref{TH:GW-well-defined}.
Recall our assumption that $R$ is local and $\frakq$ is its maximal ideal.

\begin{proof}[Proof \eqref{EQ:what-to-prove} when $\frakp\subseteq \frakq$.]

Recall that we are given $(V,f)\in \Herm[\veps]{A(\frakp),\sigma(\frakp);\tilde{k}(\frakp)}$, 
and we wish to prove
that
\[(s'_{e+1}\dd_e s'^{-1}_e[V,f])_{\frakq}=-\partial_{\frakp,\frakq}[V,f].\]
The proof is divided into three steps.
In the first step, we  construct an object
of $\Herm[\veps]{\bDer[e]{A}/\bDer[e+1]{A},D_e, \omega_e}$ 
representing $s'^{-1}_e[V,f]$.
In the second step, we use that object to describe
$\partial_{\frakp,\frakq}[V,f]$,
and in the third step we evaluate $s'_{e+1}\dd_e s'^{-1}_e[V,f]$
and prove the desired equality.

As in Section~\ref{sec:GW-Balmer}, we abbreviate $\bDer[e]{A}$ to $\calD_e$.
Recall that the shifted triangulated hermitian structures
of  $\calD':=\bDer{ A_\frakp }$  
are denoted $(D'_n,\delta'_n,\omega'_n)_{n\in\Z}$.

\medskip

\Step{1} 
Choose an $A $-lattice $U$ in $V$
such that $ U^f \frakm\subseteq U\subseteq U^f$;
see Section~\ref{sec:GW-second-res}.
We write $f_0$ for the map
$U\to U^{\tilde{*}}=\Hom_{A_S}(U,A\otimes \tilde{S} )$
given by $(f  u)v=\hat{f}(u,v)$. Then, up to natural identifications, 
$(f_0)_\frakp=f$.

\medskip

\Step{1.1}
By Lemma~\ref{LM:equiv-of-L-and-B}, 
the $L_A$-transfer functor,
$(F,j):(\calB,D_e,\omega_e)\to (\calL,\tilde{*},\tilde{\omega})$, 
is
a $1$-hermitian equivalence.
We may therefore assume that $(U,f_0)$ is given as the $L_A$-transfer of a 
some (non-unimodular)
$\veps$-hermitian space $(Y,g )$ over $(\calB,D_e,\omega_e)$.

By the definition of $\calB$,
we have $(Y_\frakp,g_\frakp)\in \Herm[\veps]{\calC^0(A_\frakp),D'_e,\omega'_e}$.
We claim that $\tilde{F}^{(K')}(Y_\frakp,g_\frakp)$ (notation as in \eqref{EQ:dfn-of-Fp}, \ref{item:K-prime})
is isomorphic to $(V,f)$.
Indeed, by Lemma~\ref{LM:derived-cat-base-change} and \ref{item:K-prime},
\[\Hom_{\calD'}(K'_A,Y_\frakp)=\Hom_{\calD'}((L_A)_\frakp,Y_\frakp)\cong \Hom(L_A,Y)_\frakp=U_\frakp=V.\]
Noting that $\Phi^{(K')}=(\Phi^{(L)})_\frakp$ (up to natural identifications),
it is routine to check that under the above isomorphism, $\tilde{F}^{(K')}(Y_\frakp,g_\frakp)=(V,f)$.

\medskip

\Step{1.2}
Let $Z$   be the cone of $g:Y\to D_eY$ in $\calD$. We claim that $Z\in \calC^0(A_\frakq)$,
i.e.,
$Z$ is supported in degrees $0,\dots,e+1$, $\HH_i(Z)=0$ for $i\neq 0$, and $\HH_0(Z)\in \rproj{A(\frakq)}$.

To see this, observe that we have a commutative diagram whose horizontal maps are isomorphisms:
\begin{equation}\label{EQ:FY-U-identification}
\xymatrix{
\HH_0(Y) \ar[r]^{\sim} \ar[d]^{\HH_0(g)} &
FY \ar@{=}[r] \ar[d]^{Fg} &
U \ar@{=}[r] \ar[d]^{f_0} &
U \ar@{^{(}->}[d] \\
\HH_0(D_eY) \ar[r]^{\sim} &
FD_eY \ar[r]^{j_Y} &
U^{\tilde{*}} \ar[r]^{f^{-1}} &
U^f
}
\end{equation}
Here, the horizontal maps of the left square are induced by the natural isomorphism $\HH_0\to F$
of Lemma~\ref{LM:F-H0-equiv-lem}, the middle square commutes because
$(U,f)$ is the $L_A$-transfer of $(Y,g)$ (Step~1.1, Lemma~\ref{LM:equiv-of-L-and-B}),
and the third square commutes because $f_0=f|_U$.
By our choice of $U$, the map $U\to U^f$ is injective and its cokernel is an $A(\frakq)$-module.
Thus, the same applies to $\HH_0(Y)\to \HH_0(D_eY)$. Now, by inspecting the long homology
exact sequence associated to the exact triangle $Y\to D_eY\to Z\to TY$, we see that
$\HH_i(Z)=0$ for all $i\neq 0$ and $\HH_0(Z)\cong \coker(U\to U^f)\in \rproj{A(\frakq)}$.
Since $Y$ and $D_eY$ are supported in  degrees $0,\dots,e$, the complex
$Z$ is supported in degrees $0,\dots,e+1$, so $Z\in \calC^0(A_\frakq)$.

\medskip

\Step{1.3}
We next claim that $(Y,g)$ represents a class in $W_\veps^e(\calD_e/\calD_{e+1})$
which maps onto $[Y_\frakp,g_\frakp]$ in $\bigoplus_{\frakt\in R^{(e)}} W^e_\veps(\bDer[e]{A_\frakt})$
(see Proposition~\ref{PR:derived-localization}).

Since $(Y,g)$ clearly maps to $(Y_\frakp,g_\frakp)$ in 
$\Herm[\veps]{\bDer[e]{A_\frakp},D'_e,\omega'_e}$, proving the claim amounts to showing that
\begin{enumerate}[label=(\roman*)]
	\item $g:Y\to D_eY$ is an isomorphism in  $\calD_e/\calD_{e+1}$,
and
	\item for all  $\frakt\in R^{(e)}-\{\frakp\}$, the class $[Y_\frakt,g_\frakt]$ is trivial
in $W_\veps^e(\bDer[e]{A_\frakt})$.
\end{enumerate}

By the definition of $\calD_e/\calD_{e+1}$ (see \cite[\S2]{Balmer_2002_Gersten_Witt_complex},
for instance), the morphism $g:Y\to D_eY$
in $\calD_e$ represents an isomorphism in $\calD_e/\calD_{e+1}$  if
its cone lives in $\calD_{e+1}$. The latter follows from  Step~1.2, so (i) holds.

To prove (ii), let $\frakt\in R^{(e)}-\{\frakp\}$. For all $i\in \Z-\{0\}$, we have $\HH_i(Y_\frakt)\cong
\HH_i(Y)_\frakt=0$. Since $\frakp- \frakt\neq \emptyset$
(otherwise $\frakp\subseteq\frakt$) 
and $\HH_0(Y)\cdot \frakp\cong U\cdot \frakp =0$,
we also have $\HH_0(Y_\frakt)\cong \HH_0(Y)_\frakt =0$. It follows that $Y_\frakt$
is isomorphic to the zero object in $\bDer[e]{A_\frakt}$, and thus
$[Y_\frakt,g_\frakt]=0$ in  $W_\veps^e(\bDer[e]{A_\frakt})$.

\medskip

{\noindent\it Conclusion of Step~1.} By putting Steps~1.1 and 1.3
together, we get   $s'_e[Y,g]=[V,f_\frakp]$ (see Construction~\ref{CN:GW-isomorphism}).

\medskip

\Step{2}
Since $U$ was chosen such that $U^f\frakm \subseteq U\subseteq U^f$,
we have $f(U^f,U^f)\subseteq A\otimes \tilde{\frakm}^{-1}$ (see Section~\ref{sec:GW-second-res}).
Write   
$\hat{f}_1=f|_{U^f\times U^f}:U^f\times U^f\to A\otimes\tilde{\frakm}^{-1}$.
Then   $\what{\partial f}:U^f/U\times U^f/U\to A\otimes\tilde{k}(\frakq)$
is given by
\[
\what{\partial f} (x+U ,y+U )=\calT(\hat{f}_1(x,y))
\qquad (x,y\in U^f).
\]
We now transform the presentation of the 
pairing $(U^f/U,\what{\partial f})$ into one defined by means of $(Y,g)$.
We shall freely view $U^{\tilde{*}}=\Hom_{A_S}(U,A\otimes\tilde{S})$ as an $A$-submodule
of $\Hom_{A(\frakp)}(V,A\otimes \tilde{k}(\frakp))$ via the localization-at-$\frakp$ map.

\medskip

\Step{2.1}
The diagram \eqref{EQ:FY-U-identification} gives rise to an isomorphism   $f^{-1}\circ j_Y:FD_eY\to U^f$
(here, $f^{-1}$ is restricted to a map from $U^{\tilde{*}}$ to $U^f$).
We claim that under this isomorphism,
the form $\hat{f}_1$ corresponds to the form $\hat{f}_2$
given by
\begin{align*}
\hat{f}_2(\vphi,\psi) 
=(\Phi_{\frakp}^{(L)})^{-1}(D'_e\vphi_\frakp\circ \veps D'_e g_\frakp^{-1}\circ \psi_\frakp)
\end{align*}
for all $\vphi,\psi\in FD_eY=\Hom(L_A,D_eY)$.
Indeed, 
\begin{align*}
\hat{f}_1(f^{-1}(j_Y \vphi) ,f^{-1} (j_Y \psi) )
&=\big(f(f^{-1}(j_Y \vphi))\big)(f^{-1} (j_Y \psi)) 
=(j_{Y_\frakp} \vphi_\frakp)(f^{-1}(j_{Y_\frakp}\psi_\frakp)) \\
&=(\Phi_{\frakp}^{(L)})^{-1}(D'_e\vphi_\frakp\circ \omega'_{e,Y_\frakp}\circ (f^{-1}(j_{Y_\frakp}\psi_\frakp)))\\
&=(\Phi_{\frakp}^{(L)})^{-1}(D'_e\vphi_\frakp\circ \omega'_{e,Y_\frakp}\circ g_\frakp^{-1}\circ \psi_\frakp)
=\hat{f}_2(\vphi,\psi).
\end{align*}
In the fourth equality we used
the fact that $f^{-1}\circ j_{Y_\frakp}=(F g)_\frakp^{-1}$,
see \eqref{EQ:FY-U-identification}.

\medskip

\Step{2.2} 
Since $U^f\frakm \subseteq U$,
every right $A$-module homomorphism
$A_S\to U^f$
restricts to a homomorphism $ A \otimes \frakm\to U$.
Recalling that $\HH_0:\calA\to \rmod{A_S}$
is an equivalence  taking 
$L_A$ to $A_S$,
$J_A$ to $A\otimes \frakm$,
and the morphism $g:Y\to D_eY$ 
to a morphism isomorphic to $U\embeds U^f$
--- see   \eqref{EQ:FY-U-identification} ---,
this means
that
for every $\vphi\in \Hom(L_A,D_eY)$,
there exists a unique $\psi\in \Hom(J_A,Y)$
such that $g\circ \psi=\vphi\circ \iota $ (with $\iota:J_A\to L_A$ 
defined in \ref{item:J-L-K-A-triangle}).
Write $\res \vphi:=\psi$. Then
\begin{align}\label{EQ:res-dfn}
g\circ \res\vphi=\vphi\circ \iota   
\end{align}
for all $\vphi\in FD_eY$.

Let $\Phi^{(J)}$ be as in Lemma~\ref{LM:A-tilde-S-diag}.
We claim that for all $\vphi,\psi\in FD_eY$, 
\[
\hat{f}_2(\vphi,\psi)=(\Phi^{(J)})^{-1}(\veps D_e(\res \vphi)\circ \psi).
\]
Since $U^f$ and $A\otimes\tilde{\frakm}^{-1}$
are  $S$-torsion-free, it is enough to check this after localizing at $\frakp$.
We check this using  Lemma~\ref{LM:A-tilde-S-diag}:
\begin{align*}
\hat{f}_2(\vphi,\psi)
&=(\Phi^{(L)}_{\frakp})^{-1}(D'_e\vphi_\frakp\circ \veps D'_e g_\frakp^{-1}\circ \psi_\frakp)\\
&=(\Phi^{(J)}_\frakp)^{-1}(D'_e\iota_\frakp \circ D'_e\vphi_\frakp\circ \veps D'_e g_\frakp^{-1}\circ \psi_\frakp)\\
&=(\Phi^{(J)}_\frakp)^{-1}(D'_e(g_\frakp \circ (\res \vphi)_\frakp  )\circ \veps D'_e g_\frakp^{-1}\circ \psi_\frakp)\\
&=(\Phi^{(J)}_\frakp)^{-1}(D_e(\res \vphi)_\frakp\circ D'_e g_\frakp \circ \veps D'_e g_\frakp^{-1}\circ \psi_\frakp)
= 
(\Phi^{(J)}_\frakp)^{-1}(\veps D_e(\res \vphi)\circ \psi)_\frakp.
\end{align*}

\medskip

\Step{2.3} 
Recall from Step~1.2 that $Z$ denotes the cone of $g:Y\to D_eY$. In particular, we have
a distinguished triangle
\begin{equation}\label{EQ:exact-Y-Z-triangle}
Y\xrightarrow{g} D_eY\xrightarrow{u} Z\xrightarrow{v} TY.
\end{equation}
in $\calD$. 
We observed in Step~1.2 that $Z\in \calC^0(A_\frakq)$
and the sequence $0\to\HH_0(Y)\to\HH_0(D_eY)\to \HH_0(Z)\to 0$
is exact. 
By Lemma~\ref{LM:F-H0-equiv-lem}, this means that $0\to FY\to FD_eY\to FZ\to 0$
is an exact sequence of right $A$-modules 
and   $p^\diamond:\Hom(K_A,Z)\to \Hom(L_A,Z)=FZ$
is an isomorphism.
Thus, we have an exact sequence
\begin{equation}\label{EQ:KAZ-exact-seq}
0\to FY\xrightarrow{Fg} FD_eY\xrightarrow{(p^{\diamond})^{-1}\circ Fu} \Hom(K_A,Z)\to 0.
\end{equation}
The diagram \eqref{EQ:FY-U-identification}
specifies
an isomorphism between the $A$-module homomorphism  $U\embeds U^f$
and $Fg:FY\to FD_eY$, and so we get an   
isomorphism of $A(\frakq)$-modules
\[\Hom(K_A,Z)\to U^f/U .\]
Unfolding the definitions, this isomorphism is evaluated
as follows. Given $\alpha\in \Hom(K_A,Z)$, 
the exactness of \eqref{EQ:KAZ-exact-seq} implies that there exists
$\vphi\in FD_eY=\Hom(L_A,D_eY)$ such that $\alpha=(p^\diamond)^{-1}((Fu)\vphi)$,
or rather, $u\circ\vphi = \alpha\circ p$ (note that $Fu=u_\diamond$).
Since the isomorphism $FD_eY\to U^f$
is $f^{-1}\circ j_Y$, the 
element of $U^f/U$ corresponding to $\alpha$ is the image of $ f^{-1}( j_Y\vphi)\in U^f$
in $U^f/U$.

\medskip

\Step{2.4} The pairing $\what{\partial f}:U^f/U\times U^f/U\to A\otimes \tilde{k}(\frakq)$
induces an $A\otimes \tilde{k}(\frakq)$-valued pairing on $\Hom(K_A,Z) $ via
the isomorphism $\Hom(K_A,Z)\to U^f/U$ of Step~2.3. Denote this pairing as  $\what{\partial f}_2$ 
and  
the corresponding $\veps$-hermitian form by $\partial f_2$.
We finish Step~2 by giving a way to evaluate 
\[
\what{\partial f}_2:\Hom(K_A,Z)\times \Hom(K_A,Z)\to A\otimes \tilde{k}(\frakq),
\]
thus describing $\partial_{\frakp,\frakq}[V,f]$ in terms of $(Y,g)$ and the
distinguished triangle \eqref{EQ:exact-Y-Z-triangle}.

Let $\alpha,\beta\in \Hom(K_A,Z)$. We showed in Step~2.3 that there are
$\vphi,\psi\in FD_eY$ such that 
\begin{align}\label{EQ:alpha-psi-rel}
&u\circ \vphi =  
\alpha\circ p,\\
&u\circ \psi =   
\beta \circ p, \nonumber
\end{align}
and the images of $\alpha$ and $\beta$ in $U^f/U$
are represented by $ f^{-1}( j_Y\vphi)$ and $ f^{-1}( j_Y\psi)$,
respectively. Thus,
\[\what{\partial f}_2(\alpha,\beta)=\what{\partial f}(f^{-1}( j_Y\vphi) ,
f^{-1}( j_Y\psi) ) =
\calT (\hat{f}_1(f^{-1}( j_Y\vphi) ,f^{-1}( j_Y\psi) )).\]
By Steps~2.1 and~2.2, the right hand side evaluates to 
$\calT((\Phi^{(J)})^{-1}(\veps D_e(\res \vphi)\circ \psi))$, so by Lemma~\ref{LM:A-tilde-S-diag},
we get
\begin{align*}
\what{\partial f}_2(\alpha,\beta)&=\calT((\Phi^{(J)})^{-1}(\veps D_e(\res \vphi)\circ \psi))\\
&=
\veps((\Phi^{(K)})^{-1}\circ (p^\diamond)^{-1}\circ (D_{e+1}q)_\diamond)(D_e(\res\vphi)\circ \psi)\\
&=\veps (\Phi^{(K)})^{-1}  (p^\diamond)^{-1}  (D_{e+1}q\circ D_e(\res\vphi)\circ \psi).
\end{align*}

\medskip

\Step{3}
We finally describe $s'_{e+1}\dd _es'^{-1}_e[V,f_\frakp]=s'_{e+1}\dd_e[Y,g]$
and check that its negative agrees with $\partial_{\frakp,\frakq}[V,f_\frakp]$,
which is represented by $(\Hom(K_A,Z), {\partial f}_2)$
defined Step~2.4.

\medskip

\Step{3.1} 
Recall the distinguished triangle \eqref{EQ:exact-Y-Z-triangle}.
By 
\cite[Theorem~2.6]{Balmer_2000_Triangular_Witt_I},
we have a commutative diagram  
\begin{equation}\label{EQ:map-between-triangles}
\xymatrix{
Y \ar[r]^{g } \ar[d]^{\delta_e\veps\omega_e} & 
D_eY \ar[r]^u \ar@{=}[d] &
Z \ar[r]^v \ar@{.>}[d]^{h} &
TY  \ar[d]^{\delta_e\veps T\omega_e} \\
D_eD_eY \ar[r]^{\delta_e D_e g} &
D_eY \ar[r]^{-TD_ev} &
D_{e+1}Z \ar[r]^{TD_eu} &
TD_eD_eY
}
\end{equation}
in which the top and buttom rows are distinguished triangles in $\calD $
and   $h$ is an isomorphism satisfying $h=\veps\omega_{e+1}\circ D_{e+1}h$.
We saw in Step~1.2 that $Z\in \calC^0(A_\frakq)\subseteq \calD_{e+1}$.
Thus, 
according to \cite[Definitions~5.16, 2.10]{Balmer_2000_Triangular_Witt_I},
$(Z,h)$ represents the image of $[Y,g]$ under the map $\partial_e: W^e_\veps(\calD_e/\calD_{e+1})\to 
W^{e+1}_\veps(\calD_{e+1})$
(see Section~\ref{sec:GW-Balmer}). Consequently, $(Z,h)$ represents
$\dd_e[Y,g]\in W^{e+1}(\calD_{e+1}/\calD_{e+2})$.  (In fact, $\calD_{e+2}=0$ because
we assume that $R$ is local with maximal ideal $\frakq$.)
Since $ Z\in\calC^0(A_\frakq)$,
it follows that $(Z,h)=(Z_\frakq,h_\frakq)$ 
represents the image of $\dd_e[Y,g]$ in $W_\veps(\calC^0(A_\frakq),D_e,\omega_e)$,
see Construction~\ref{CN:GW-isomorphism}. Consequently, with notation
as in \eqref{EQ:dfn-of-Fp} (and $K$ as in \ref{item:K}),  we have 
\[\tilde{F}^{(K)}[Z,h] = s'_{e+1}\dd_e[Y,g]=s'_{e+1}\dd_e s'^{-1}_e[V,f_\frakp].\]

Unfolding the definitions, $\tilde{F}^{(K)}(Z,h)$ is the $\veps$-hermitian space $(\Hom(K_A,Z),h_1)$,
where $\hat{h}_1:\Hom(K_A,Z)\times \Hom(K_A,Z)\to A\otimes\tilde{k}(\frakq)$ is given by
\[
\hat{h}_1(\alpha,\beta)
=(\Phi^{(K)})^{-1}(D_{e+1}\alpha\circ \veps h\circ \beta)
\]
for all $\alpha,\beta\in \Hom(K_A,Z)$.

\medskip

\Step{3.2} We now 
show that $\hat{h}_1 = -\what{\partial f}_2$. By
Steps~2.4 and~3.1, this would imply
that   $s'_{e+1}\dd_e s'^{-1}_e[V,f_\frakp]=-\partial_{\frakp,\frakq}[V,f_\frakp]$,
completing the proof.

Let $\alpha,\beta\in \Hom(K_A,Z)$, and let $\vphi,\psi\in \Hom(L_A,Y)$ be elements
satisfying \eqref{EQ:alpha-psi-rel}. By Steps~2.4 and~3.1, checking
that $\hat{h}_1(\alpha,\beta) = -\what{\partial f}_2(\alpha,\beta)$
amounts to showing that
\[
(\Phi^{(K)})^{-1}(D_{e+1}\alpha\circ \veps h\circ \beta)
=- \veps (\Phi^{(K)})^{-1}  (p^\diamond)^{-1}  (D_{e+1}q\circ D_e(\res\vphi)\circ \psi ),
\]
or equivalently, that
\begin{equation}\label{EQ:final-eq}
D_{e+1}\alpha\circ \veps h\circ \beta\circ p
= -
D_{e+1}q\circ D_e(\res\vphi)\circ \psi .
\end{equation}

Consider the   diagram
\[
\xymatrix{
J_A \ar[r]^\iota \ar[d]^{\res \vphi} &
L_A \ar[r]^{p} \ar[d]^{\vphi} & 
K_A \ar[r]^{q} \ar@{.>}[d]^{\alpha'} & 
TJ \ar[d]^{T(\res \vphi)} \\
Y \ar[r]^g &
D_e Y \ar[r]^u &
Z \ar[r]^v &
TY
}
\]
in which the left square is \eqref{EQ:res-dfn},
and the top and bottom rows are the distinguished triangles
\ref{item:J-L-K-A-triangle} and~\eqref{EQ:exact-Y-Z-triangle}.
Since $\calD$ is a triangulated category,
there exists $\alpha'\in\Hom(K_A,Z)$ which makes the diagram commute. 
By the middle square and \eqref{EQ:alpha-psi-rel},
we have $p^\diamond \alpha' =u\circ \vphi=p^\diamond \alpha$.
Lemma~\ref{LM:F-H0-equiv-lem} tells us that $p^\diamond:\Hom(K_A,Z)\to \Hom(L_A,Z)$
is bijective, so 
$\alpha'=\alpha$.
Plugging this into the right square of the diagram gives
$T(\res\vphi)\circ q= v\circ \alpha$.
Applying $D_{e+1}$ to both sides, we get  
\[D_{e+1}q\circ D_e(\res \vphi)=D_{e+1}\alpha\circ D_{e+1} v.\]
Using this, the middle square of \eqref{EQ:map-between-triangles}, and \eqref{EQ:alpha-psi-rel}, 
we get
\begin{align*}
-D_{e+1}q\circ D_e(\res \vphi)\circ \psi
&= -D_{e+1}\alpha\circ D_{e+1} v\circ \psi= D_{e+1}\alpha\circ h\circ u \circ \psi\\
&=D_{e+1}\alpha\circ h \circ \beta \circ p,
\end{align*}
so we proved \eqref{EQ:final-eq}. This completes the proof. 
\end{proof}

\section{Azumaya Algebras of Index $2$}
\label{sec:exactness-more}

Let $(A,\sigma)$ be an Azumaya algebra with involution
over a regular ring $R$, and let $\veps\in\mu_2(R)$.
We finish by
applying Theorem~\ref{TH:GW-well-defined} to prove  the exactness of $\aGW{A,\sigma,\veps}$
when $R$ is semilocal of dimension $\leq 3$, $\ind A=2$ and $\sigma$ is orthogonal or symplectic.
The proof will combine the exact octagon
of \cite{First_2022_octagon},
which was shown to be compatible with the differentials of $\aGW{A,\sigma,\veps}$
in \cite[\S6]{Bayer_2019_Gersten_Witt_complex_prerprint},
together with  the \emph{Gersten--Witt spectral
sequence}, introduced by Balmer and Walter \cite{Balmer_2002_Gersten_Witt_complex} 
in the case $(A,\sigma,\veps)=(R,\id_R,1)$,
and by Gille \cite{Gille_2007_hermitian_GW_complex_I}, \cite{Gille_2009_hermitian_GW_complex_II}
in general. Theorem~\ref{TH:GW-well-defined} is what allows
us to use both of these tools at the same time.

\begin{thm}\label{TH:exactness-dim-three-ind-two}
	Assume $R$ is a regular semilocal domain of dimension $\leq 3$.
	If $\ind A\leq 2$ and $\sigma$ is orthogonal or symplectic, then
	$\aGW{A,\sigma,\veps}$ is exact.
\end{thm}

\begin{proof}
	By \cite[Theorem~8.7]{Bayer_2019_Gersten_Witt_complex_prerprint},
	it is enough to consider the following two cases:
	\begin{enumerate}[label=(\arabic*)]
		\item $\deg A=1$;
		\item $\deg A=2$,
		and there are $\lambda,\mu\in \units{A}$
		such that $\lambda^\sigma=-\lambda$, $\mu^\sigma=-\mu$, $\lambda\mu=-\mu\lambda$
		and $\lambda^2\in \Cent(A)$.
	\end{enumerate} 
	Since $\sigma$ is orthogonal or symplectic, $\Cent(A)=R$,
	and so
	case (1) is  covered
	by \cite[Corollary~10.4]{Balmer_2002_Gersten_Witt_complex}
	and \cite[p.~3]{Balmer_2005_shifted_Witt_groups_semilocal}.
	We proceed with case (2).

	Define   $B,\tau_1,\tau_2$
	as in \cite[\S6]{Bayer_2019_Gersten_Witt_complex_prerprint}. Then $B$ is a quadratic
	\'etale $R$-subalgebra of $A$, $\tau_1$ is its standard $R$-involution
	and
	$\tau_2=\id_B$. Moreover,   $A=B\oplus \mu B$ and $B=\Cent(B)=R\oplus \lambda R$,
	so the set $\{1,\lambda,\mu,\mu\lambda\}$ is an $R$-basis of $A$.
	Since $\sigma$ maps $\lambda$, $\mu$ and $\mu\lambda$ to their negatives,
	$\sigma$ is necessarily symplectic (but $\veps$ can be $1$ or $-1$). 
	
	By \cite[Theorem~6.2]{Bayer_2019_Gersten_Witt_complex_prerprint},
	the data of $A,\sigma,\lambda,\mu$ gives rise to an $8$-periodic exact 
	sequence of Gersten-Witt  complexes as defined in Section~\ref{sec:GW-second-res}.
	Since   $\aGW{B,\tau_2,-1}=0$, it degenerates
	into a $7$-term exact sequence:
	\begin{align}\label{EQ:seven-term-seq}
	0  \to \,&
	\aGW{A,\sigma,1}\to
	\aGW{B,\tau_1,1} \to
	\aGW{A,\sigma,-1} \to
	\aGW{B,\tau_2,1} \to \\
	&\aGW{A,\sigma,-1}\to
	\aGW{B,\tau_1,-1} \to 
	\aGW{A,\sigma,1} \to	
	0 .\nonumber
	\end{align}
	We view this sequence as a double cochain complex.
	The columns $\aGW{B,\tau_1,\pm 1}$ and $\aGW{B,\tau_2,1}$
	are exact by \cite[Theorem~9.4]{Bayer_2019_Gersten_Witt_complex_prerprint}, and we
	have $\HH^{3}(\aGW{A,\sigma,\pm 1})=0$
	by  \cite[Theorem~5.1]{Bayer_2019_Gersten_Witt_complex_prerprint},
	or \cite[Theorem~8.4]{Gille_2020_hermitian_Springer} (applies
	to $\iaGW{A,\sigma,\pm1}$) and Theorem~\ref{TH:GW-well-defined}.
	
	Recall
	that $\GW{A,\sigma,\veps}$ denotes $\aGW{A,\sigma,\veps}$
	with the $-1$-term removed.
	The Gersten--Witt spectral sequence
	associated to $(A,\sigma)$, illustrated in Figure~\ref{FG:spectral-seq-of-GW}, 
	converges to $\{W^n_{+1}(A,\sigma)\}_{n\in\Z}$ and thus
	gives rise to an exact sequence
	\[
	0\to \HH^2(\GW{A,\sigma,\pm 1})\to W_{\mp1}(A,\sigma) \xrightarrow{(*)}\HH^0(\GW{A,\sigma,\mp1})\to 0
	\]
	in which $(*)$ is the localization map   $W_{\mp 1}(A,\sigma)\to \HH^0(\GW{A,\sigma,\mp 1})
	\subseteq \GW{A,\sigma,\mp 1}_0$, i.e., $\dd_{-1}$. 
	This means
	that $\HH^0(\aGW{A,\sigma,\pm 1})=0$.
	Note that
	here and in Figure~\ref{FG:spectral-seq-of-GW}, 
	we have implicitly used the isomorphism $\iaGW{A,\sigma,\veps}\cong \aGW{A,\sigma,\veps}$
	of Theorem~\ref{TH:GW-well-defined}.

	Now, repeated application of \cite[Lemma~9.2]{Bayer_2019_Gersten_Witt_complex_prerprint}
	to \eqref{EQ:seven-term-seq} 
	(which is essentially diagram chasing)	
	shows that $\aGW{A,\sigma, \pm 1}$
	is exact. Specifically, apply this lemma to assert the nullity
	of all cohomology groups in the following order:
	$\HH^{-1}(\aGW{A,\sigma,  1})=0$, $\HH^{1}(\aGW{A,\sigma,   1})=0$,
	$\HH^{2}(\aGW{A,\sigma,  1})=0$,
	$\HH^{-1}(\aGW{A,\sigma, - 1})=0$,
	$\HH^{1}(\aGW{A,\sigma,  -1})=0$,
	$\HH^{2}(\aGW{A,\sigma, - 1})=0$.
\end{proof}

\begin{figure}[h]
	\[
	\begin{array}{llllll}
	\vdots & \vdots & \vdots & \vdots &   \\
	\HH^0(\GW{A,\sigma,-1}) & \HH^1(\GW{A,\sigma,-1}) & \HH^2(\GW{A,\sigma,-1}) &  0 & \cdots \\
	0 & 0 & 0 &   0& \cdots \\
	\HH^0(\GW{A,\sigma,1}) & \HH^1(\GW{A,\sigma,1}) & \HH^2(\GW{A,\sigma,1}) &  0& \cdots \\
	0 & 0 & 0 &   0~~& \cdots \\
	\HH^0(\GW{A,\sigma,-1}) & \HH^1(\GW{A,\sigma,-1}) & \HH^2(\GW{A,\sigma,-1}) & 0& \cdots \\
	\vdots & \vdots & \vdots & \vdots & 
	\end{array}
	\]
	\caption{The $E_2$ page of the Gersten--Witt spectral sequence of $(A,\sigma)$.
	The term $\HH^0(\GW{A,\sigma,1})$ is in the $(0,0)$-coordinate.}
	\label{FG:spectral-seq-of-GW}
	\end{figure}

\begin{remark}\label{RM:why-both}
	It is not straightforward to 
	prove Theorem~\ref{TH:exactness-dim-three-ind-two} directly using only
	the
	Gersten--Witt complex $\iaGW{A,\sigma,\veps}$ of Balmer--Walter 
	and Gille. 
	Letting $A,\sigma,\tau_1,\tau_2$ be as in the proof of Theorem~\ref{TH:exactness-dim-three-ind-two},
	the functoriality of $\iaGW{A,\sigma,\veps}$ with
	respect to hermitian functors does give rise
	to a $7$-term sequence of cochain complexes
	analogous to \eqref{EQ:seven-term-seq}:
	\begin{align}\label{EQ:seven-term-seq-B}
	0  \to \,&
	\iaGW{A,\sigma,1}\to
	\iaGW{B,\tau_1,1} \to
	\iaGW{A,\sigma,-1} \to
	\iaGW{B,\tau_2,1} \to \\
	&\iaGW{A,\sigma,-1}\to
	\iaGW{B,\tau_1,-1} \to 
	\iaGW{A,\sigma,1} \to	
	0  \nonumber
	\end{align}
	but it is {\it a priori} not clear
	that \eqref{EQ:seven-term-seq-B} 
	is exact in levels $1$ and above, and the exactness
	at all levels is needed in the last paragraph of the proof. 
	
	In more detail,
	let $e\geq 0$  and let us identify
	$\iaGW{A,\sigma,\veps}$ with $\bigoplus_{\frakp\in R^{(e)}}\tilde{W}_\veps(A(\frakp))\cong
	\bigoplus_{\frakp\in R^{(e)}}W_\veps(A(\frakp),\sigma(\frakp))$ via
	the isomorphism constructed by Balmer--Walter \cite{Balmer_2002_Gersten_Witt_complex}
	for $A=R$, or Gille \cite{Gille_2007_hermitian_GW_complex_I}, \cite{Gille_2009_hermitian_GW_complex_II}
	in general.
	One   can show that  the $e$-th level of \eqref{EQ:seven-term-seq-B}
	decomposes as a direct sum of $7$-term   sequences $\bigoplus_{\frakp\in R^{(e)}}S(\frakp)$. 
	It is expected  that for each $\frakp\in R^{(e)}$, the sequence  $S(\frakp)$ is isomorphic to
	the exact sequence 
	\begin{align}\label{EQ:seven-term-seq-P}
	0 & \to   
	W_1(A(\frakp),\sigma(\frakp))\to
	W_1(B(\frakp),\tau_1(\frakp)) \to
	W_{-1}(A(\sigma),\sigma(\frakp)) \to
	W_1(B(\frakp),\tau_2(\frakp))\\ 
	&\to   W_{-1}(A(\frakp),\sigma(\frakp)) \to
	W_{-1}(B(\frakp),\tau_1(\frakp)) \to 
	W_1(A(\frakp),\sigma(\frakp)) \to	\nonumber
	0   
	\end{align}
	considered in \cite[Corollary~8.2]{First_2022_octagon} and earlier
	in Lewis \cite{Lewis_1982_improved_exact_sequences}, but when $e\geq 1$,
	it is not  immediate that the maps in $S(\frakp)$
	correspond to the maps appearing in \eqref{EQ:seven-term-seq-P},
	so
	we cannot assert that
	$S(\frakp)$ is exact without further work.
	(In fact, by crunching through   Construction~\ref{CN:GW-isomorphism},
	it is possible
	to show that \eqref{EQ:seven-term-seq} and \eqref{EQ:seven-term-seq-B}
	are isomorphic and as a consequence assert the exactness of  \eqref{EQ:seven-term-seq-B}
	in levels $1$ and above.)
\end{remark}

\bibliographystyle{plain}
\bibliography{MyBib_18_05}

\end{document}

%% file: amsart_header_18_05.tex
\usepackage{amsfonts}
\usepackage{amssymb}
\usepackage{amsmath}
\usepackage{hyperref}
\usepackage{mathrsfs}
\usepackage{centernot}
\usepackage{mathdots}
\usepackage{stmaryrd}
\usepackage[all]{xy}

\usepackage{mathtools}

\usepackage{color}

\newtheorem{thm}{Theorem}[section]
\newtheorem{lem}[thm]{Lemma}
\newtheorem{prp}[thm]{Proposition}
\newtheorem{cor}[thm]{Corollary}

\newtheoremstyle{roman} 
    {8.0pt plus 2.0pt minus 4.0pt}                    
    {8.0pt plus 2.0pt minus 4.0pt}                    
    {\normalfont}                
    {}                           
    {\bfseries}                  
    {.}                          
    {5pt plus 1pt minus 1pt}     
    {}  

\theoremstyle{roman}

\newtheorem{construction}[thm]{Construction}
\newtheorem{example}[thm]{Example}
\newtheorem{remark}[thm]{Remark}

\theoremstyle{plain}

\newcommand{\Step}[1]{\noindent {\it Step #1.}} 

\newcommand{\rem}[1]{}

\newcommand{\N}{\mathbb{N}}

\newcommand{\Z}{\mathbb{Z}}

\newcommand{\HH}{{\mathrm{H}}}

\newcommand{\fraka}{{\mathfrak{a}}}
\newcommand{\frakb}{{\mathfrak{b}}}

\newcommand{\frakm}{{\mathfrak{m}}}

\newcommand{\frakp}{{\mathfrak{p}}}
\newcommand{\frakq}{{\mathfrak{q}}}

\newcommand{\frakt}{{\mathfrak{t}}}

\newcommand{\calA}{{\mathcal{A}}}
\newcommand{\calB}{{\mathcal{B}}}
\newcommand{\calC}{{\mathcal{C}}}
\newcommand{\calD}{{\mathcal{D}}}

\newcommand{\calH}{{\mathcal{H}}}

\newcommand{\calK}{{\mathcal{K}}}
\newcommand{\calL}{{\mathcal{L}}}
\newcommand{\calM}{{\mathcal{M}}}

\newcommand{\calP}{{\mathcal{P}}}

\newcommand{\calT}{{\mathcal{T}}}

\newcommand{\bfG}{{\mathbf{G}}}

\newcommand{\catC}{{\mathscr{C}}}
\newcommand{\catD}{{\mathscr{D}}}

\newcommand{\what}[1]{\widehat{#1}}

\newcommand{\veps}{\varepsilon}
\newcommand{\vphi}{\varphi}

\newcommand{\onto}{\rightarrow\mathrel{\mkern-14mu}\rightarrow} 
\newcommand{\embeds}{\hookrightarrow}

\newcommand{\suchthat}{\,:\,}
\newcommand{\where}{\,|\,}

\newcommand{\quo}[1]{\overline{#1}}

\DeclareMathOperator{\Cent}{Cent} %
\DeclareMathOperator{\coker}{coker} %
\DeclareMathOperator{\End}{End} %
\DeclareMathOperator{\Ext}{Ext} %
\DeclareMathOperator{\hgt}{hgt} %
\DeclareMathOperator{\Hom}{Hom} %
\DeclareMathOperator{\id}{id} %
\DeclareMathOperator{\im}{im} %
\DeclareMathOperator{\ind}{ind} %
\newcommand{\op}{\mathrm{op}} %
\DeclareMathOperator{\Spec}{Spec} %
\newcommand{\et}{\mathrm{\acute{e}t}}

\newcommand{\units}[1]{{#1^\times}}

